\documentclass[11pt,reqno,a4paper]{amsart}
\usepackage[utf8]{inputenc}
\usepackage[doi=true,isbn=false,style=alphabetic,sorting=nyt,backend=biber,maxnames=10,maxalphanames=3,giveninits=true]{biblatex}
\AtEveryBibitem{\clearlist{language}}
\bibliography{main}

\usepackage[breaklinks,colorlinks,plainpages,hypertexnames=false,plainpages=false]{hyperref}
\hypersetup{urlcolor=blue, citecolor=blue, linkcolor=blue}

\usepackage{amsmath}
\usepackage{amsthm}
\usepackage{amssymb}
\usepackage{amsfonts}
\usepackage{bbm}
\usepackage{booktabs}
\usepackage{mathtools}
\usepackage{tikz-cd}
\usepackage{hhline}
\usepackage{stmaryrd}
\usepackage{enumitem}
\usepackage{centernot}
\usepackage{mathrsfs}
\usepackage{comment}
\usepackage{listings}
\usepackage{multicol}
\usepackage[capitalise, noabbrev]{cleveref} 
\usepackage[noend]{algorithmic} 

\usepackage{mathdots}
\usepackage{cleveref} 
\usepackage{nicematrix} 
\usepackage{mathalpha} 
\DeclareMathAlphabet{\dutchcal}{U}{dutchcal}{m}{n}

\usepackage{array}
\newcolumntype{L}[1]{>{\raggedright\let\newline\\\arraybackslash\hspace{0pt}}m{#1}}
\newcolumntype{C}[1]{>{\centering\let\newline\\\arraybackslash\hspace{0pt}}m{#1}}
\newcolumntype{R}[1]{>{\raggedleft\let\newline\\\arraybackslash\hspace{0pt}}m{#1}}

\newcommand{\ben}[1]{\textcolor{magenta}{Ben: #1}}

\newcommand{\ollie}[1]{\textcolor{olive!60!green}{Ollie: #1}}

\usepackage{fullpage}

\graphicspath{{../}}  

\usepackage{tikz}
\usetikzlibrary{arrows}
\usetikzlibrary{decorations.pathreplacing}
\usetikzlibrary{matrix}
\usetikzlibrary{calc}
\usetikzlibrary{shapes}
\usetikzlibrary{patterns}
\usetikzlibrary{fit,backgrounds,scopes}

\newtheoremstyle{theoremstyle}
{10pt}      
{5pt}       
{\itshape}  
{}          
{\bfseries} 
{}         
{ }      
{}          

\newtheoremstyle{algorithmstyle}
{10pt}      
{5pt}       
{}  
{}          
{\bfseries} 
{}         
{ }      
{}          

\newtheoremstyle{examplestyle}
{10pt}      
{5pt}       
{}          
{}          
{\bfseries} 
{}         
{ }      
{}          

\makeatletter
\newtheorem*{rep@theorem}{\rep@title}
\newcommand{\newreptheorem}[2]{%
\newenvironment{rep#1}[1]{%
 \def\rep@title{#2 \ref{##1}}%
 \begin{rep@theorem}}%
 {\end{rep@theorem}}}
\makeatother

\makeatletter 
\newcommand{\subalign}[1]{%
  \vcenter{%
    \Let@ \restore@math@cr \default@tag
    \baselineskip\fontdimen10 \scriptfont\tw@
    \advance\baselineskip\fontdimen12 \scriptfont\tw@
    \lineskip\thr@@\fontdimen8 \scriptfont\thr@@
    \lineskiplimit\lineskip
    \ialign{\hfil$\m@th\scriptstyle##$&$\m@th\scriptstyle{}##$\hfil\crcr
      #1\crcr
    }%
  }%
}
\makeatother

\newreptheorem{theorem}{Theorem}
\newreptheorem{corollary}{Corollary}
\newreptheorem{proposition}{Proposition}

\theoremstyle{theoremstyle}
\newtheorem{theorem}{Theorem}[section]
\newtheorem{lemma}[theorem]{Lemma}
\newtheorem{proposition}[theorem]{Proposition}
\newtheorem{corollary}[theorem]{Corollary}
\newtheorem{conjecture}[theorem]{Conjecture}
\newtheorem*{theorem*}{Main theorem}
\newtheorem*{lemma*}{Lemma}
\theoremstyle{examplestyle}
\newtheorem{example}[theorem]{Example}
\newtheorem{definition}[theorem]{Definition}
\newtheorem*{notation*}{Notation}
\newtheorem{remark}[theorem]{Remark}

\theoremstyle{algorithmstyle}

\newcommand{\NN}{\mathbb{N}}
\newcommand{\CC}{\mathbb{C}}
\newcommand{\RR}{\mathbb{R}}

\newcommand{\ZZ}{\mathbb{Z}}

\newcommand{\suchthat}{\;\ifnum\currentgrouptype=16 \middle\fi|\;}
\newcommand{\bigmid}{\left.\vphantom{\Big\{} \suchthat \vphantom{\Big\}}\right.}

\newcommand{\cB}{\mathcal{B}}
\newcommand{\cC}{\mathcal{C}}
\newcommand{\cF}{\mathcal{F}}

\newcommand{\cH}{\mathcal{H}}
\newcommand{\cN}{\mathcal{N}}

\newcommand{\tF}{\widetilde{F}}
\newcommand{\tH}{\widetilde{H}}
\newcommand{\nbc}[1]{{\rm nbc}(#1)}

\DeclareMathOperator{\cl}{cl}
\DeclareMathOperator{\codim}{codim}
\DeclareMathOperator{\conv}{conv}
\DeclareMathOperator{\cone}{cone}

\DeclareMathOperator{\mult}{mult}
\DeclareMathOperator{\Newt}{Newt}

\DeclareMathOperator{\Span}{Span}
\DeclareMathOperator{\Supp}{Supp}

\DeclareMathOperator{\Trop}{Trop}
\DeclareMathOperator{\val}{val}
\DeclareMathOperator{\Lam}{Lam}
\DeclareMathOperator{\rank}{rank}

\DeclareMathOperator{\rec}{rec}

\DeclareRobustCommand{\Chi}{{\mathpalette\irchi\relax}}
\newcommand{\irchi}[2]{\raisebox{\depth}{$#1\chi$}}

\tikzstyle{edge}=[line width=1.5pt,black]
\tikzstyle{vertex}=[fill=black,circle,inner sep=0pt, minimum size=4pt]

\makeatletter
\let\@fnsymbol\@arabic
\makeatother

\title{A tropical approach to rigidity:\\ counting realisations of frameworks}
\author{Oliver Clarke}
\address{Oliver Clarke, Department of Mathematical Sciences, Durham University.}
\email{oliver.clarke@durham.ac.uk}
\urladdr{https://www.oliverclarkemath.com/}
\author{Sean Dewar}
\address{Sean Dewar, School of Mathematics, University of Bristol.}
\email{sean.dewar@bristol.ac.uk}
\urladdr{https://www.seandewar.com/}
\author{Daniel Green Tripp}
\address{Daniel Green Tripp, School of Mathematics, University of Bristol.}
\email{daniel.greentripp@bristol.ac.uk}
\urladdr{https://sites.google.com/view/daniel-greentripp/}
\author{James Maxwell}
\address{James Maxwell, School of Mathematics, University of Bristol.}
\email{james.william.maxwell@gmail.com}
\urladdr{https://sites.google.com/view/jmacademicsite/}
\author{\\ Anthony Nixon}
\address{Anthony Nixon, School of Mathematical Sciences, Lancaster University.}
\email{a.nixon@lancaster.ac.uk}
\urladdr{https://www.lancaster.ac.uk/maths/people/anthony-nixon}
\author{Yue Ren}
\address{Yue Ren, Department of Mathematical Sciences, Durham University.}
\email{yue.ren2@durham.ac.uk}
\urladdr{https://www.yueren.de/}
\author{Ben Smith}
\address{Ben Smith, School of Mathematical Sciences, Lancaster University.}
\email{b.smith9@lancaster.ac.uk}
\urladdr{https://sites.google.com/view/bsmithmathematics/}
\date{}

\begin{document}

\begin{abstract}
A realisation of a graph in the plane as a bar-joint framework is rigid if there are finitely many other realisations, up to isometries, with the same edge lengths. Each of these finitely-many realisations can be seen as a solution to a system of quadratic equations prescribing the distances between pairs of points. For generic realisations, the size of the solution set depends only on the underlying graph so long as we allow for complex solutions. We provide a characterisation of the realisation number -- that is the cardinality of this complex solution set -- of a minimally rigid graph. Our characterisation uses tropical geometry to express the realisation number as an intersection of Bergman fans of the graphic matroid. As a consequence, we derive a combinatorial upper bound on the realisation number involving the Tutte polynomial. Moreover, we provide computational evidence that our upper bound is usually an improvement on the mixed volume bound.
\end{abstract}

\maketitle

\section{Introduction}

A \emph{$d$-dimensional bar-joint framework}  $(G,p)$ is an ordered pair consisting of a graph $G$ and a map $p:V\rightarrow \mathbb{R}^d$. For brevity we will simply use framework if the dimension is implicit. The map $p$ is often referred to as a \emph{realisation} of $G$.
The framework $(G,p)$ is \emph{rigid} if all edge-length preserving continuous deformations of the vertices are isometries of $\mathbb{R}^d$. The study of rigidity is classical, dating back to work of Cauchy \cite{Cau} and Euler \cite{euler} on convex polytopes.

Determining if a framework is rigid is computationally challenging \cite{Abbot} and hence most recent works have focussed on the generic case.
In the generic case, determining if a framework is rigid reduces to a purely graph-theoretic property that can be characterised by the rank of a matrix associated to the graph \cite{AsimowRothI}.
Even for non-generic frameworks, the predictions from the generic analysis have been applied to a myriad of real-world practical applications including computer-aided geometric design \cite{CAD}, stability of mechanical structures \cite{ConnellyGuest} and modelling for crystals and other materials \cite{Periodic,Theran_etal}.

It is a well known fact \cite{AsimowRothI} that either almost all $d$-dimensional frameworks of a given graph are rigid, or almost all are not rigid;
if the former holds, we say that the graph $G$ is \emph{$d$-rigid}.
Furthermore, $G$ is \emph{minimally $d$-rigid} if it is $d$-rigid and $G-e$ is not for any edge $e$ of $G$.
An example of a 2-rigid graph is the unique graph obtained from the complete graph on four vertices by deleting an edge. This will be denoted by $K_4^-$, see \Cref{fig:k4-e}.
It is folklore that minimally 1-rigid graphs are exactly trees, and an elegant combinatorial description of minimally 2-rigid graphs was provided by Pollaczek-Geiringer \cite{PollaczekGeiringer}.

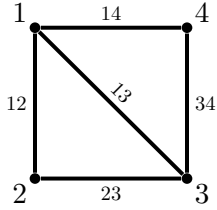
\begin{figure}[t]
	\begin{center}
        \begin{tikzpicture}[scale=1]
        \begin{scope}
			\node[vertex] (1) at (0,1) {};
			\node[vertex] (2) at (0,-1) {};
			\node[vertex] (3) at (2,-1) {};
			\node[vertex] (4) at (2,1) {};

			\node (v1) at (-0.2,1.2) {1};
			\node (v2) at (-0.2,-1.2) {2};
			\node (v3) at (2.2,-1.2) {3};
			\node (v4) at (2.2,1.2) {4};

			\draw[edge] (1) -- (2) node[midway, left, scale=.7] {12};
			\draw[edge] (1) -- (3) node[midway,above,sloped,scale=.7] {13};
			\draw[edge] (1) -- (4) node[midway,above,scale=.7] {14};
			\draw[edge] (2) -- (3) node[midway,below,scale=.7] {23};
			\draw[edge] (3) -- (4) node[midway,right,scale=.7] {34};
        \end{scope}
		\end{tikzpicture}
	\end{center}
 \caption{$K_4^-$ is a minimally 2-rigid graph with 2-realisation number 2.}\label{fig:k4-e}
\end{figure}

Another natural question is as follows. Given a graph $G$ and a generic realisation $(G,p)$ in $\mathbb{R}^d$, how many other realisations $(G,q)$ in $\mathbb{R}^d$ have the same edge lengths? If the graph is $d$-rigid, then up to isometric transformations this number is finite and called the \emph{$d$-realisation number}. Determining the $d$-realisation number has applications in sensor network localisation \cite{Sensors}, conformation change in biological structures \cite{Biol} and control of autonomous systems of robots \cite{Formations}.

Theoretically, the $d$-realisation number can be obtained by symbolic computation techniques such as Gr\"{o}bner basis based algorithms. However these are computationally intractable in practice. Asymptotic upper bounds were computed as complex bounds of the determinantal variety of the Euclidean distance matrix by Borcea and Streinu \cite{BS04}. In the case when $d=2$, mixed volume techniques have also been used \cite{Stef10}. Jackson and Owen \cite{JO19} established the 2-realisation number, or bounds on it, for several families of graphs including planar graphs and graphs whose rigidity matroid is connected.

An improvement on symbolic Gr\"obner basis computations are probabilistic numerical algorithms that randomly sample from the space of frameworks. However, despite the speed gain, they are still quite slow and provide only a lower bound to the realisation number.
Currently, the best algorithm known appears in \cite{cggkls} and has an implementation at \cite{TropicalAlg}. The algorithm uses tropical geometry and a recursive deletion-contraction construction on an auxiliary combinatorial object called a \emph{bigraph}. Due to the recursive nature of the algorithm, the time complexity is still exponential.

We will approach the realisation number problem using tropical geometry, a combinatorial analogue of algebraic geometry.
Its early successes were within enumerative algebraic geometry, where it was utilised as a combinatorial method for computing invariants such as Gromov-Witten invariants~\cite{Mikhalkin05} or (re-)proving formulas such as the Caporaso-Harris formula \cite{GathmannMarkwig2007}. 
More recently, it has played a vital role in the development of intersection theory for matroids, in which the \emph{Bergman fan} of a matroid can be viewed as a tropical variety.
These innovations were key to solving a number of outstanding conjectures on log-concavity within matroid theory~\cite{AHK:18}.

Tropical geometry has already found applications in rigidity theory. As previously mentioned, the realisation number algorithm of \cite{cggkls} uses tropical geometry working over the field of Puiseux series. Bernstein and Krone \cite{Bernstein} analysed the tropical Cayley-Menger variety to give a new proof of Pollaczek-Geiringer's characterisation of minimally 2-rigid graphs. Ideas from tropical geometry, via valuation theory, have also been used to understand when 2-rigid graphs have flexible realisations \cite{GLS2019}.
In the other direction, rigidity has recently found applications in tropical geometry. In particular, infinitesimal rigidity and the Maxwell-Cremona correspondence were used to understand extremal decompositions for tropical varieties \cite{Farhad}.

\medskip
\noindent \textbf{Our approach.}
Given a minimally $d$-rigid graph $G$, we study the generic number of solutions, or the \textit{generic root count}, of the \textit{edge-length} and \textit{vertex-pinning} polynomials for $G$ described in \Cref{def:polynomialSystemVertexVariables}. We show that the generic root count is equal to the $d$-realisation number of $G$ and then modify these equations to use \textit{edge-length variables} as in \Cref{def:polynomialSystemEdgeVariables}. In the case when $d = 2$, we perform a sequence of modifications to the polynomials in the edge-length variables to prove our main result, \Cref{thm:twoDimensionalRealisationNumber}, which shows that the $2$-realisation number of a minimally $2$-rigid graph is a tropical intersection product of: a Bergman fan $\Trop(M_G)$, its negative $-\Trop(M_G)$, and a hyperplane, where $M_G$ is the graphic matroid of $G$.

We use the description of the $2$-realisation number as a tropical intersection product to give a combinatorial characterisation in \Cref{thm:arboreal}, which allows us to give bounds in terms of matroid invariants. The \textit{non-broken circuit bases} (nbc-bases) of a matroid $M_G$ are a special subset of bases, see \Cref{def:nbc-basis}, whose size is equal to the evaluation of the Tutte polynomial $T_{M_G}(1,0)$. See \Cref{cor:alt+bound} for some alternative combinatorial descriptions for this value. We show that the number of nbc-bases is equal to the tropical intersection product of the negative Bergman fan $-\Trop(M_G)$ of $G$, the Bergman fan of the uniform matroid $\Trop(U_{m, m-k+1})$, and a tropical hyperplane $\Trop(y_\epsilon - 1)$. Observe that this tropical intersection product is obtained by replacing one copy of the graphic matroid in the tropical intersection product in \Cref{thm:twoDimensionalRealisationNumber} with the uniform matroid $U_{m, m-k+1}$. Intuitively, the Bergman fan of the uniform matroid is bigger than the graphic matroid $M_G$, hence this replacement cannot increase the tropical intersection product, i.e., the number nbc-bases gives an upper bound on the $2$-realisation number. In \Cref{subsec: computations}, we observe that for minimally rigid graphs with at most $10$ vertices, the number of nbc-bases provides a significantly more accurate upper bound than the mixed volume bound.

\medskip
\noindent \textbf{Statement of main result.}
Our main theorem describes the 2-realisation number of a graph $G$, here denoted by $c_2(G)$, by a tropical intersection product involving the Bergman fan $\Trop(M_G)$ and its `flip' $-\Trop(M_G)$.
Here we use $X \cdot Y \cdot Z$ to represent the tropical intersection product of tropical varieties $X,Y, Z$ (see \Cref{def:tropicalIntersectionProduct}),
and $\Trop (f)$ to represent the tropical hypersurface of the polynomial $f$. 

\begin{theorem*}[\Cref{thm:twoDimensionalRealisationNumber}]
  Let $G=([n],E)$ be a minimally 2-rigid graph with $n \geq 3$ vertices and an edge $e\in E$.
  Then the following equality holds:
  \[  c_2(G) = \frac{1}{2}(-\Trop(M_G))\cdot \Trop(M_G)\cdot \Trop(y_e-1). \]
\end{theorem*}

The final term of the tropical intersection in \Cref{thm:twoDimensionalRealisationNumber} may be interpreted as `pinning' a single edge into place.
This is a common approach for computing 2-realisation numbers as it removes all isometries except the single reflection which fixes in place the pinned edge.

While we regard Bergman fans $\Trop(M)$ as fans in Euclidean space $\RR^n$ for the sake of consistency with general tropical varieties,  it is not uncommon to regard them as fans $\overline{\Trop(M)}$ in the tropical torus $\RR^n/(1,\dots,1)\cdot \RR$; see for example \cite[Section 4.2]{MaclaganSturmfels15}.  Using said notation, the statement in \cref{thm:main} can be simplified to
\begin{equation*}
   c_2(G) = \frac{1}{2}(-\overline{\Trop(M_G)})\cdot \overline{\Trop(M_G)},
\end{equation*}
where ``\,$\cdot$\,'' instead denotes the number of points counted with multiplicity in the stable intersection in $\RR^n/(1,\dots,1)\cdot \RR$.

\subsection{Outline}

In \Cref{sec:prelim}, we give the necessary preliminaries from rigidity theory and tropical geometry, including about realisation numbers, generic root counts, and matroid theory.
\Cref{sect:numberViaTropInt} builds up to the key result, \Cref{thm:main}, that expresses the 2-realisation number of a minimally 2-rigid graph as a tropical intersection product.
In \Cref{sec:realisations_via_matroid_intersection}, we use \Cref{thm:main} to deduce a number of results on the 2-realisation number.
The first main result is \Cref{thm:arboreal}, a combinatorial characterisation in terms of chains of flats of the graphic matroid of $G$.
The second main result is \Cref{thm:nbc+bases+upper+bound}, an upper bound on $c_2(G)$ in terms of the number of nbc-bases of $M_G$, or equivalently, as an evaluation of the Tutte polynomial of $G$.
We end with \Cref{cor:lower+bound}, a combinatorial lower bound on $c_2(G)$ in terms of special bases of $M_G$.


\section{Preliminaries}
\label{sec:prelim}

In this section, we cover the preliminaries for rigidity theory, algebraic geometry, and tropical geometry needed throughout the paper.
First, we remark on the notion of `genericity' used in this paper.

Given an algebraic set $S$ over a field $K$, we say that a property $P$ of the points in $S$ holds for \emph{almost all} points of $S$ or \emph{generic} points of $S$, if $P$ holds for all points in some non-empty Zariski open subset of $S$.
For $K=\RR$ or $K=\CC$ specifically, the definition of `generic' used here is a strictly weaker notion than what is usually used in most rigidity theory literature. There, a point $(x_i)_{i \in [n]} \in \CC^n$ is `generic' if $x_1,\dots,x_n$ are algebraically independent over $\mathbb{Q}$.

\subsection{Rigidity preliminaries}

For any positive integer $n$, we denote the first $n$ positive integers by $[n] := \{1,\ldots,n\}$, and we denote the set of all $2$-tuples of distinct elements in $[n]$ by $\binom{[n]}{2}$. Throughout the entire article, $G$ will be a simple undirected graph with vertex set $[n]$ and edge set $E(G) \subseteq \binom{[n]}{2}$. We denote by $K_n$ the complete graph on $[n]$.

It is a well-established fact in algebraic geometry that understanding the real solutions of a set of equations is difficult, while understanding the complex solutions is (although admittedly still difficult) easier.
With this in mind, we instead wish to `complexify' our concept of rigidity.
To do so,
we define the \emph{(complex) rigidity map} to be the multivariable map
\begin{equation*}
    f_{G,d}\colon \mathbb{C}^{n\cdot d} \rightarrow \mathbb{C}^{|E(G)|} , \quad \big(p_{i,k}\big)_{\subalign{i&\in[n]\\ k&\in[d]}} \mapsto \frac{1}{2}\left(\sum_{k=1}^d(p_{i,k}-p_{j,k})^2 \right)_{ij\in E}.
\end{equation*}
We define any point $p \in \mathbb{C}^{n\cdot d}$ to be a \emph{realisation} of $G$,
with $p_i = (p_{i,k})_{k \in [d]}$ representing the position of vertex $i$.
If there is a need to distinguish between whether a realisation lies in $\mathbb{R}^{n\cdot d}$ or in $\mathbb{C}^{n\cdot d}$,
we will explicitly state that it is either a \emph{real realisation} or a \emph{complex realisation} of $G$ respectively.

Given $O(d, \mathbb{C})$ is the group of $d \times d$ complex-valued matrices $M$ where $M^T M = M M^T = I$,
we define two realisations $p,q \in \mathbb{C}^{n \cdot d}$ to be \emph{congruent} (denoted by $p \sim q$) if and only if there exists $A \in O(d, \mathbb{C})$ and $x \in \mathbb{C}^{d}$ so that $p_i = Aq_i + x$ for all $i \in [n]$.
If the set of vertices of $(G,p)$ affinely span $\mathbb{C}^d$,
we have the following equivalent statement:
two realisations $p,q$ are congruent if and only if $f_{K_n,d}(p) = f_{K_n,d}(q)$ (see \cite[Section 10]{Gortler2014} for more details).

We can now give an alternative characterisation of generic rigidity using our complexified system of constraint equations (i.e., the map $f_{G,d}$) using the following result.

\begin{proposition}[see, for example, {\cite[Proposition 3.5]{dewar2024number}}]\label{prop:rigidequiv}
    For a graph $G$ with $n \geq d+1$, the following are equivalent:
    \begin{enumerate}
        \item $G$ is $d$-rigid;
        \item for almost all realisations $p$, the set $f^{-1}_{G,d} (f_{G,d}(p))/\sim$ is finite;
        \item for almost all points $\lambda$ in the Zariski closure of $f_{G,d}(\CC^{n \cdot d}
        )$, the set $f_{G,d}^{-1}(\lambda)/\sim$ is finite.
    \end{enumerate}
\end{proposition}

Equivalently to the definition given in the introduction, we can say that
a $d$-rigid graph $G$ is minimally $d$-rigid if the Jacobian of $f_{G,d}$ at a generic point of $\CC^{n \cdot d}$ is surjective.
This condition can be characterised using our complexified set-up with the following equivalence which follows from \cite[Lemma 3.1]{dewar2024number}.

\begin{proposition}\label{prop:minrigidequiv}
    For a graph $G$ with $n \geq d+1$, the following are equivalent:
    \begin{enumerate}
        \item $G$ is minimally $d$-rigid;
        \item $G$ is $d$-rigid and $f_{G,d}$ is \emph{dominant}, i.e., the Zariski closure of $f_{G,d}(\CC^{n \cdot d}
        )$ is $\mathbb{C}^{|E(G)|}$.
    \end{enumerate}
\end{proposition}

To define a classical necessary condition for $d$-rigidity, we require the following terminology.
For natural numbers $k$ and $\ell$, a graph $G=(V,E)$ is \emph{$(k,\ell)$-sparse} if every subgraph $(V',E')$ with $|V'| \geq k$ has $|E'|\leq k|V'|-\ell$. Further it is \emph{$(k,\ell)$-tight} if $|E|=k|V|-\ell$ and it is $(k,\ell)$-sparse.

\begin{lemma}[{\cite[Lemma 11.1.3]{Whiteley:1996}}]\label{lem:maxwelllike} 
    Let $G$ be a $d$-rigid graph with $|V| \geq d$.
    Then $G$ contains a spanning subgraph which is $(d, \binom{d+1}{2})$-tight.
    If $G$ is minimally $d$-rigid, then $G$ is $(d, \binom{d+1}{2})$-tight.
\end{lemma}

Pollaczek-Geiringer \cite{PollaczekGeiringer} characterised minimally 2-rigid graphs as precisely the $(2,3)$-tight graphs (sometimes called \emph{Laman graphs} in the literature). Her result was independently rediscovered by
Laman \cite{Laman} and, as a result, is often referred to as Laman's theorem.

\begin{theorem}[Pollaczek-Geiringer \cite{PollaczekGeiringer}]
    A graph with at least two vertices is minimally 2-rigid if and only if it is $(2,3)$-tight.
\end{theorem}

\subsection{Realisation numbers}

The following result describes more explicitly what exactly `finite' means in \Cref{prop:rigidequiv}.

\begin{proposition}[see, for example, {\cite[Proposition 3.5]{dewar2024number}}]\label{prop:realisationnum}
    Let $G$ be a $d$-rigid graph with at least $d+1$ vertices.
    Then there exists a positive integer $c_d(G)$ so that for any generic realisation $p$, the set $f^{-1}_{G,d} (f_{G,d}(p))/\sim$ contains exactly $c_d(G)$ points.
\end{proposition}

We now define the value $c_d(G)$ given in \Cref{prop:realisationnum} to be the \emph{$d$-realisation number} of a $d$-rigid graph $G$.

In practice, we are actually more interested in understanding the cardinality of the set $f^{-1}_{G,d} (f_{G,d}(p))/\sim$ when we restrict to real realisations.
It is easy to see that this value is no longer a generic value however -- see \Cref{fig:realrealisations} -- which makes it more challenging to investigate compared to $c_d(G)$.
Fortunately, we can always use the $d$-realisation number of the graph to bound the number of equivalent real realisations (modulo congruences) for a generic framework in $\mathbb{R}^d$.
Note though that this upper bound is not always tight:
the graph pictured in \Cref{fig:prismadapted} (first observed by Jackson and Owen \cite{JO19}) gives such an example.

\begin{figure}[tp]
	\begin{center}
        \begin{tikzpicture}[scale=0.8]
        \begin{scope}
			\node[vertex] (1) at (0,1) {};
			\node[vertex] (2) at (0,-1) {};
			\node[vertex] (3) at (1.5,-0.5) {};
			\node[vertex] (4) at (1.5,0.5) {};
			\node[vertex] (5) at (4,0) {};

            \node[anchor=south east] at (1) {1};
            \node[anchor=north east] at (2) {2};
            \node[anchor=north] at (3) {3};
            \node[anchor=south] at (4) {4};
            \node[anchor=west] at (5) {5};

			\draw[edge] (1)edge(2);
			\draw[edge] (1)edge(3);
			\draw[edge] (1)edge(4);
			\draw[edge] (2)edge(3);
			\draw[edge] (2)edge(4);
			\draw[edge] (5)edge(3);
			\draw[edge] (5)edge(4);
        \end{scope}
	\begin{scope}[xshift=210]
			\node[vertex] (1) at (0,1) {};
			\node[vertex] (2) at (0,-1) {};

			\node[vertex] (3) at (-1.5,-0.5) {};
			\node[vertex] (4) at (1.5,0.5) {};
			\node[vertex] (5) at (-0.5, 1.5) {};

            \node[anchor=south east,yshift=-2mm] at (1) {1};
            \node[anchor=north] at (2) {2};
            \node[anchor=east] at (3) {3};
            \node[anchor=west] at (4) {4};
            \node[anchor=south] at (5) {5};

			\draw[edge] (1)edge(2);
			\draw[edge] (1)edge(3);
			\draw[edge] (1)edge(4);
			\draw[edge] (2)edge(3);
			\draw[edge] (2)edge(4);
			\draw[edge] (5)edge(3);
			\draw[edge] (5)edge(4);

	\end{scope}
         \begin{scope}[xshift=300, yshift=-40]
             \draw (0,0) -- (0,3);
         \end{scope}
         \begin{scope}[xshift=350]
			\node[vertex] (1) at (0,1) {};
			\node[vertex] (2) at (0,-1) {};

			\node[vertex] (3) at (1.5,-0.5) {};
			\node[vertex] (4) at (1.5,0.5) {};
	      \node[vertex] (5) at (2,0) {};

            \node[anchor=south east,yshift=-2mm] at (1) {1};
            \node[anchor=north east] at (2) {2};
            \node[anchor=north west] at (3) {3};
            \node[anchor=south west] at (4) {4};
            \node[anchor=west] at (5) {5};

			\draw[edge] (1)edge(2);
			\draw[edge] (1)edge(3);
			\draw[edge] (1)edge(4);
			\draw[edge] (2)edge(3);
			\draw[edge] (2)edge(4);
			\draw[edge] (5)edge(3);
			\draw[edge] (5)edge(4);
        \end{scope}
		\end{tikzpicture}
	\end{center}
 \caption{On the left and the right are two realisations of the same graph in the plane with different edge lengths. Both sets of edge lengths can be chosen generically. On the left the realisation allows the equivalent realisation in the middle, on the right such a realisation is not possible by the triangle inequality, illustrating that the real realisation number depends on the specific choice of generic realisation.}\label{fig:realrealisations}
\end{figure}
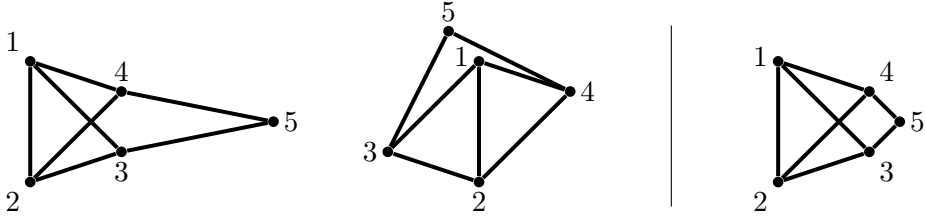

\begin{figure}[tp]
	\begin{center}
            \begin{tikzpicture}[scale=0.5]
			\node[vertex] (11) at (-1,0) {};
			\node[vertex] (21) at (-2,1) {};
			\node[vertex] (31) at (-2,-1) {};

			\node[vertex] (12) at (1,0) {};
			\node[vertex] (22) at (2,1) {};
			\node[vertex] (32) at (2,-1) {};

			\node[vertex] (a) at (-3,2) {};
                \node[vertex] (b) at (3,2) {};

			\draw[edge] (11)edge(12);
			\draw[edge] (21)edge(22);
			\draw[edge] (31)edge(32);

			\draw[edge] (11)edge(21);
			\draw[edge] (31)edge(11);

			\draw[edge] (12)edge(22);
			\draw[edge] (32)edge(12);

			\draw[edge] (a)edge(21);
			\draw[edge] (a)edge(12);
			\draw[edge] (a)edge(31);

                \draw[edge] (b)edge(22);
			\draw[edge] (b)edge(11);
			\draw[edge] (b)edge(32);

		\end{tikzpicture}
	\end{center}
	\caption{A minimally 2-rigid graph $G$ with $c_2(G_2) = 45$. As the number of real equivalent realisations to a generic realisation in $\mathbb{R}^2$ is always even (a corollary of a classical result of Hendrickson \cite{hend92}), the upper bound on this number given by $c_2(G)$ is not tight.}\label{fig:prismadapted}
\end{figure}
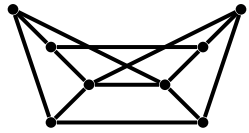

Working with the quotient space $f^{-1}_{G,d}(f_{G,d}(p))/\sim$ is difficult,
since many of the tropical techniques showcased later in the paper are suitable for counting generic fibres of polynomial maps, not their quotients.
Instead of directly quotienting out the congruences,
we can instead restrict our domain to achieve essentially the same effect.
The standard method for doing this in rigidity theory is to `pin' vertices to various points to stop any congruent motions; for example, if $d=3$, we would fix our realisations to have the first vertex fixed at the point $(0,0,0)$, the second vertex restricted to the $x$-axis $\{(x,0,0) : x \in \mathbb{C} \}$,
and the third vertex restricted to the $xy$-plane $\{(x,y,0) : x,y \in \mathbb{C}\}$.
In fact, this exact method is described in \cite[Lemma 3.3]{dewar2024number} with this pinning system.
As zero coordinates behave strangely under tropicalisation (due in part to valuations taking an infinite value at such points),
we have slightly edited this method to allow for a more general pinning system.

\begin{lemma}
\label{lem:realisationNumberAsGenericFibreCardinality}
  Let $G$ be a $d$-rigid graph with $n \geq d+1$.
  Let $b_1,\ldots,b_d \in \RR^d$ be a basis, let $(c_1,\dots,c_d)\coloneqq[b_1\cdots b_d](1,\dots,1)\in\RR^d$ where $[b_1\cdots b_d]\in\RR^{d\times d}$ is the matrix with rows $b_1,\dots,b_d$, and define $Y \subset \CC^{|E(G)|}$ to be the Zariski closure of the image of $f_{G,d}$ and
  \begin{equation*}
    X \coloneqq \left\{ p \in \mathbb{C}^{n\cdot d}\mid p_i \cdot b_l = c_l \text{ for } i \in [d] \text{ and } l \in [d+1-i]  \right\}.
  \end{equation*}
  Then the restricted rigidity map $f_{G,d}|_X^Y\colon X\rightarrow Y$ is dominant and $2^d c_{d}(G)$ is the generic cardinality of fibres of $f_{G,d}|_X^Y$, i.e,
  \[2^d c_{d}(G) = \#\Big( f_{G,d}^{-1}(\lambda) \cap X \Big) \text{ for } \lambda\in Y \text{ generic}. \]
\end{lemma}

\begin{proof}
  Note that a realisation $p\in\CC^{n \cdot d}$ lies in $X$ if and only if it's first $d$ vertices $p_1,\dots,p_d$ lie in a flag of affine spaces in $\CC^d$, namely:
  \begin{itemize}
  \item $p_1\in [b_1\cdots b_d]^{-1}(c_1,\dots,c_d)=\{(1,\dots,1)\}$,
  \item $p_2\in [b_1\cdots b_{d-1}]^{-1}(c_1,\dots,c_{d-1}),$
  \item $p_3\in [b_1\cdots b_{d-2}]^{-1}(c_1,\dots,c_{d-2}),$
  \item etc.
  \end{itemize}
  In \cite[Lemma 3.3]{dewar2024number} the statement is proved for one particular choice of flag (given in \cite[Equation (1)]{dewar2024number}).  Using the Gram-Schmidt process on $b_1,\ldots,b_d$, we can construct an isometry of $\RR^d$ (which is a map $x \mapsto Ax + b$ for some $A \in O(d,\mathbb{R}) \subset O(d,\mathbb{C})$ and $b \in \mathbb{R}^d \subset \mathbb{C}^d$) that maps said flag to ours.
  Since any such map is an isometry of the bilinear space of $\mathbb{C}^d$ equipped with the dot product,
  it follows that the cardinalities of the fibres remain the same.
\end{proof}

When $G$ is minimally $d$-rigid,
we may combine \Cref{prop:minrigidequiv} and \Cref{lem:realisationNumberAsGenericFibreCardinality} to replace the algebraic set $Y$ with the linear space $\mathbb{C}^{|E(G)|}$.
This greatly simplifies the computational techniques that are required.

\subsection{Matroidal preliminaries}\label{sec:matroidalPreliminaries}

We briefly recall the necessary preliminaries on matroids; for further details, see \cite{Oxley:2011}.

Let $M = (E, r)$ be a matroid on ground set $E$ with rank function $r \colon 2^E \to \ZZ_{\geq 0}$.
A set $X \subseteq E$ is \emph{independent} if $r(X) = |X|$, and \emph{dependent} otherwise.
The \emph{bases} of $M$ are the maximal independent sets and the \emph{circuits} of $M$ are the minimal dependent sets.
A \emph{flat} $F \subseteq E$ of $M$ is an inclusion-maximal subset of $E$ of a fixed rank, i.e. $r(F \cup \{e\}) = r(F) + 1$ for all $e \in E - F$.
Flats can also be defined as the closed sets of $M$ under the closure operator
\[
\cl\colon 2^E \rightarrow 2^E \, , \quad \cl(A) = \{e \in E \mid r(A \cup e) = r(A) \} \, .
\]
The main family of matroids we will be concerned with will be those arising from graphs.


\begin{example}[Graphic matroids]\label{ex:graphic+matroid}
    Let $G = (V,E)$ be a graph.
    Its \emph{graphic matroid} $M_G$ is the matroid on the ground set $E$ where the rank $r(F)$ of $F \subseteq E$ is the size of a spanning forest of $G[F]$, the subgraph of $G$ with edge set $F$ and vertex set $\{v \in V(G) \mid  \exists \: vw \in F \}$.
    The flats of $M_G$ correspond to vertex-disjoint unions of the vertex-induced subgraphs of $G$. 
    The circuits of $M_G$ are the cycles of $G$.
    If $G$ is connected, the bases of $M_G$ are the spanning trees of $G$.
\end{example}

A \emph{loop} is an element of a matroid contained in no basis.
The flats of a \emph{loopless} matroid (i.e. containing no loop) form a lattice ordered by inclusion with $\emptyset$ as the minimal flat and $E$ as the maximal flat.
A \emph{chain of flats} $\cF = (F_1, \dots, F_s)$ is a chain in this lattice, i.e. $F_i \subsetneq F_j$ for all $i < j$.
We call a chain \emph{proper} if it does not include $\emptyset$ or $E$, as any proper chain can always be extended to include these elements.
We call a proper chain \emph{maximal} if for any flat $F'$ such that $F_i \subseteq F' \subseteq F_{i+1}$, we have either $F' = F_i$ or $F' = F_{i+1}$.
It follows that maximal chains have exactly $r(M) - 1$ parts with $r(F_i) = i$.
Given two chains of flats $\cF = (F_1, \dots, F_s)$ and $\cH = (H_1, \dots, H_t)$, we say $\cF$ \emph{refines} $\cH$ if every $H_i$ appears in $\cF$.
We denote the set of proper chains of flats of a matroid $M$ by $\Delta(M)$\footnote{This is sometimes also known as the \emph{order complex} of $M$, specifically the order complex of the lattice of flats minus $\emptyset$ and $E$.}.

We can encode the matroid $M$ via a polyhedral fan whose geometry reflects the combinatorics of $M$.
To each proper chain of flats $\cF = (F_1, \dots, F_s) \in \Delta(M)$, we associate a polyhedral cone $\sigma_\cF \subseteq \RR^E$.
Write $\chi_e \in \RR^E$ for the characteristic vector of $e \in E$, and for each $S \subseteq E$ define $\chi_S := \sum_{e \in S} \chi_e \in \RR^E$.
We write $\cone(x_1,\ldots,x_n)$ for the set of sums $\sum_{i=1}^n a_i x_i$ with $a_i \geq 0$ for each $i \in [n]$, and define
\begin{equation} \label{eq:bergman+cone}
\sigma_{\cF} = \cone(\chi_{F_1}, \dots, \chi_{F_s}) + \RR \cdot \chi_E \, .
\end{equation}
As $\{\chi_{F_1}, \dots, \chi_{F_s}, \chi_E\}$ are linearly independent, it follows that $\sigma_{\cF}$ is a simplicial cone of dimension $s+1$.
It is also straightforward to check that $\sigma_\cH \subseteq \sigma_\cF$ if and only if $\cF$ refines $\cH$.
The \emph{Bergman fan} $\Trop(M)$ associated to the loopless matroid $M$ is the union of the cones $\sigma_{\cF}$ for all proper chains of flats $\cF \in \Delta(M)$\footnote{In other literature, the Bergman fan is defined as a fan with the same support but equipped with the coarser \textit{matroid fan structure}; see \cite{ArdilaKlivans:2006}. Our default fan structure on the Bergman fan, with cones $\sigma_{\mathcal F}$ as above, is sometimes called the \textit{fine structure} or \textit{fine subdivision}.}.
The previous properties of $\sigma_\cF$ imply it is a pure simplicial polyhedral fan of dimension $r(E)$.
The definition of Bergman fan above is obtained by unpacking \cite[Theorem~1]{ArdilaKlivans:2006}.
For further properties of this fan, see \cite{ArdilaKlivans:2006, Feichtner2005}.

\begin{example}\label{ex:bergman}
    The \emph{uniform matroid} $U_{n,s}$ is the matroid on ground set $[n]$ whose bases are all subsets of $[n]$ of cardinality $s$.
    The flats of $U_{n, s}$ are $[n]$ and all sets of cardinality less than $s$, hence the maximal chains of flats are of the form
    \[
    \cF_{\mu} \colon \{\mu(1)\} \subsetneq \{\mu(1), \mu(2)\} \subsetneq \cdots \subsetneq \{\mu(1),\dots,\mu(s-1)\} \, 
    \]
    where $\mu \in \rm{Sym}(n)$ is a permutation.
    As such, the Bergman fan $\Trop(U_{n,s})$ is the union of the maximal cones $\sigma_{\cF(\mu)}$ for all $\mu \in \rm{Sym}(n)$, where
\begin{align}
\sigma_{\cF(\mu)} &= \cone(\chi_{\{\mu(1)\}}, \chi_{\{\mu(1), \mu(2)\}}, \dots, \chi_{\{\mu(1),\dots,\mu(s-1)\}}) + \RR\cdot \chi_{[n]}  \nonumber\\
&= \{w \in \RR^n \mid w_{\mu(1)} \geq w_{\mu(2)} \geq \cdots \ge w_{\mu(s-1)} \geq w_{\mu(s)} = \cdots = w_{\mu(n)}\} \, . \label{eq:bergman+cone+ex}
\end{align}
\end{example}

\subsection{Tropical preliminaries}

We will briefly recall the notions of tropical varieties, stable intersections, and how generic root counts can be expressed as tropical intersection numbers.  Our notation closely follows \cite{MaclaganSturmfels15}, except that our tropical varieties will be \emph{weighted polyhedral complexes} -- polyhedral complexes with positive integer multiplicities attached to the maximal cells.
If we wish to solely consider the set of points in the tropical variety with no additional polyhedral structure, we will refer to the \emph{support} of the tropical variety.
For the sake of simplicity, we will focus on the classes of tropical varieties that are of immediate interest to us, namely tropical hypersurfaces and tropical linear spaces.
For a more rigorous treatment of general tropical varieties, see \cite{MaclaganSturmfels15}.

Throughout, we let $K$ be an algebraically closed field with valuation map $\val \colon K \rightarrow \RR \cup \{\infty\}$.
We say the valuation is \emph{trivial} if it only takes values $0$ and $\infty$, and non-trivial otherwise.

\begin{example}
    Let $\CC\{\!\{t\}\!\}$ denote the field of Puiseux series
    \[
    \CC\{\!\{t\}\!\} = \left\{ \sum_{k=k_0}^\infty c_kt^{\frac{k}{n}} \; \Big| \; n \in \NN \, , \, k_0 \in \ZZ \, , \, c_k \in \CC \right\} \,\]
    with valuation
    \[\quad \val\left(\sum_{k=k_0}^\infty c_kt^{\frac{k}{n}}\right) = \frac{k_0}{n} \, .
    \]
    The field $\CC\{\!\{t\}\!\}$ is algebraically closed and $\val$ is a non-trivial valuation map.
    Note that $\CC$ is a subfield of $\CC\{\!\{t\}\!\}$ on which the valuation is trivial.
    Throughout the paper we exclusively work with the cases when $K = \CC$ or $K=\CC\{\!\{t\}\!\}$.
\end{example}

We write $K[x^\pm] := K[x_1^\pm, \dots, x_n^\pm]$ for the ring of Laurent polynomials in $n$ variables and coefficients in $K$.
Recall that the zero locus of a Laurent polynomial is contained in the algebraic torus $(K^\times)^n$.

We next define \emph{tropical hypersurfaces}.
\begin{definition}\label{def:tropical+hypersurface}
Let $f=\sum_{\alpha\in S}c_\alpha x^\alpha\in K[x^\pm]$ be a Laurent polynomial with finite support $S\subseteq\ZZ^n$.
For each subset $s \subseteq S$, we define a closed polyhedron
\[
\sigma_s(f)\coloneqq \Big\{ w\in\RR^n\mid \min_{\alpha\in S}(\val(c_\alpha)+\alpha\cdot w) \text{ is attained at exactly $\beta$ for each } \beta \in s\Big\} \subseteq \RR^n \, .
\]
The (unweighted) \emph{tropical hypersurface} $\Trop(f)$ is the polyhedral complex
\[
\Trop(f)\phantom{:}=\Big\{ \sigma_s(f)\mid s \subseteq S\text{ with }|s|>1 \, , \, \sigma_s(f) \neq \emptyset\Big\} \, .
\]
\end{definition}

We can obtain an elegant polyhedral description of $\Trop(f)$ via subdivisions of the Newton polytope.
Recall that the \emph{Newton polytope} of $f$ is $\Newt(f) := \conv(\alpha \in \ZZ^n \mid \alpha \in S) \subseteq \RR^n$.
The \emph{Newton subdivision} $\cN(f)$ of $\Newt(f)$ is the regular subdivision on $S$ induced by $\val(c_\alpha)$; see \cite[Definition 2.3.8]{MaclaganSturmfels15} for a formal definition of regular subdivision.
Informally, $\cN(f)$ is constructed by lifting the points of $S$ to the heights $\{\val(c_\alpha)\}_{\alpha \in S}$ in $\RR^{n+1}$, taking the convex hull and then projecting 
the subcomplex of faces that can be seen when viewed from below back to $\RR^n$.
Each cell in $\cN(f)$ is uniquely determined by the elements of $S$ it contains, and so we write $\tau_s(f)$ for the cell satisfying $\tau_s(f) \cap S = s$.

In \cite[Proposition 3.1.6]{MaclaganSturmfels15}, it is stated that $\Trop(f)$ is dual to $\cN(f)$ in the following way.
There is a one-to-one correspondence between the cells $\sigma_s(f)$ and the positive dimensional cells $\tau_s(f)$ of $\cN(f)$.
Moreover, this correspondence is inclusion reversing, and satisfies $\dim(\sigma_s(f)) = n - \dim(\tau_s(f))$.
As $\dim(\tau_s(f)) = 0$ if and only if $|s| = 1$, it follows that $\Trop(f)$ is a \emph{pure} polyhedral complex of dimension $n-1$ -- i.e., every maximal cell has the same dimension (which in our case is $n-1$).

We can use this correspondence between $\Trop(f)$ and $\cN(f)$ to add one further piece of information to $\Trop(f)$, namely \emph{multiplicities} to the maximal cells.
If $\sigma_s(f)$ is $(n-1)$-dimensional, the cell $\tau_s(f) \in \cN(f)$ is 1-dimensional.
Here we define $\ell(s) := |\tau_s(f) \cap \ZZ^n| - 1$ to be the lattice length of $\tau_s(f)$.
Since every 0-dimensional cell of $\cN(f)$ is contained in $\mathbb{Z}^n$,
we have $\ell(s) \geq 1$.
We define the (weighted) \emph{tropical hypersurface} to be $\Trop(f)$ with multiplicity $\ell(s)$ attached to maximal cell $\sigma_s(f)$.
From here on, we shall always assume our tropical hypersurfaces are weighted unless stated otherwise.

  \begin{example}
  \label{ex:tropicalHypersurface}
  \cref{fig:tropicalHypersurface} shows the Newton subdivision $\cN(f)$ and the tropical hypersurface $\Trop(f)$ when $f\coloneqq 1+x+y+t\cdot xy\in \CC\{\!\{t\}\!\}[x^\pm,y^\pm]$. The figure highlights two cells in $\cN(f)$ and their corresponding polyhedra $\Trop(f)$.  Note how minimal cells (of cardinality greater than one) in $\cN(f)$ correspond to maximal cells of $\Trop(f)$.
  Each of the maximal cells of $\Trop(f)$ has multiplicity one.
\end{example}

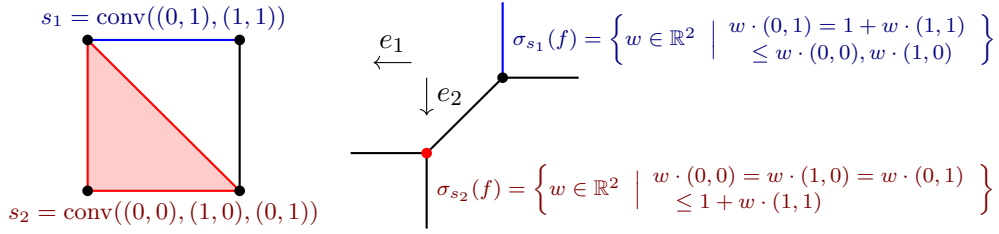
\begin{figure}[t]
  \centering
  \begin{tikzpicture}
    \node[left]
    {
      \begin{tikzpicture}[x={(20mm,0mm)},y={(0mm,20mm)}]
        \draw[thick] (1,1) -- (1,0);
        \draw[thick,blue] (1,1) -- (0,1);
        \fill[red!20,thick,draw=red] (0,0) -- (1,0) -- (0,1) -- cycle;
        \fill (0,0) circle (2pt);
        \fill (1,0) circle (2pt);
        \fill (0,1) circle (2pt);
        \fill (1,1) circle (2pt);
        \node[above,blue!50!black,font=\footnotesize] at (0.5,1) {$s_1=\conv((0,1),(1,1))$};
        \node[below,red!50!black,font=\footnotesize] at (0.5,0) {$s_2=\conv((0,0),(1,0),(0,1))$};
      \end{tikzpicture}
    };
    \node[right]
    {
      \begin{tikzpicture}
        \draw[thick]
        (0,0) -- ++(-1,0)
        (0,0) -- ++(0,-1)
        (0,0) -- (1,1)
        (1,1) -- ++(1,0);
        \draw[thick,blue]
        (1,1) -- ++(0,1);
        \fill[red] (0,0) circle (2pt);
        \fill (1,1) circle (2pt);
        \node[anchor=north west,red!50!black,font=\scriptsize] at (0,0)
        { $\sigma_{s_2}(f)=\left\{w\in\RR^2 \bigmid\hspace{-2mm}
            \begin{array}{l}
              w\cdot (0,0) = w\cdot (1,0) = w\cdot (0,1)\\ \hspace{3mm}\leq 1+w\cdot (1,1)
            \end{array} \right\}$ };
        \node[anchor=west,blue!50!black,font=\scriptsize] at (1,1.5)
        { $\sigma_{s_1}(f)=\left\{w\in\RR^2 \bigmid\hspace{-2mm}
            \begin{array}{l}
              w\cdot (0,1) = 1+w\cdot (1,1) \\ \hspace{3mm}\leq w\cdot (0,0),w\cdot (1,0)
            \end{array} \right\}$ };
        \draw [->] (-0.2,1.2) -- node[above] {$e_1$} ++(-0.5,0);
        \draw [->] (0,1) -- node[right] {$e_2$} ++(0,-0.5);
      \end{tikzpicture}
    };
  \end{tikzpicture}\vspace{-3mm}
  \caption{$\cN(f)$ and $\Trop(f)$ for the Laurent polynomial $f$ given in \Cref{ex:tropicalHypersurface}.
  }
  \label{fig:tropicalHypersurface}
\end{figure}

We are now in a position to define tropical varieties.

\begin{definition} \label{def:tropical+variety}
  Let $I\subseteq K[x^\pm]=K[x_1^\pm,\dots,x_n^\pm]$ be a Laurent polynomial ideal.
  The \emph{tropical variety} $\Trop(I)$ is the weighted polyhedral complex in $\RR^n$
  \[
  \Trop(I) = \bigcap_{f \in I} \Trop(f) \, .
  \]
  The recipe for its multiplicities is given in~\cite[Definition 3.4.3]{MaclaganSturmfels15}.
\end{definition}

\begin{remark}
    A priori, Definition~\ref{def:tropical+variety} requires taking an infinite intersection.
    However, one can always construct $\Trop(I)$ as a finite intersection of tropical hypersurfaces~\cite[Theorem 2.6.6]{MaclaganSturmfels15}.
    The tropical variety $\Trop(I)$ has a number of equivalent descriptions \cite[Fundamental Theorem 3.2.3]{MaclaganSturmfels15}.
    When $I$ is prime, $\Trop(I)$ has additional structural properties: it is pure, balanced, and connected in codimension one~\cite[Structure Theorem 3.3.5]{MaclaganSturmfels15}.
    For us, it is often important that our polyhedral complexes are balanced (\Cref{def:balanced}) as this guarantees that intersection numbers are translation invariant; see \Cref{lem:tropicalIntersectionProductTranslationInvariant}.
\end{remark}

Outside of our special cases of interest, it is sufficient for us to know that one can place multiplicities on the maximal cells.
In the case that $I = \langle f \rangle$ is a principal ideal, its tropical variety $\Trop(\langle f \rangle)$ is precisely the tropical hypersurface $\Trop(f)$.
Moreover, its multiplicities agree with those obtained from taking the lattice lengths of the 1-dimensional cells in the Newton subdivision $\cN(f)$. 

The other class of tropical varieties of interest to us are tropical linear spaces.
  For a linear ideal $I\subseteq K[x_1^\pm, \dots, x_n^\pm]$ (i.e., $I$ is generated by linear polynomials), we associate a matroid $M(I)$ to $I$ on ground set $[n]$ in the following way.
  We say a subset $S \subseteq [n]$ is dependent in $M(I)$ if there exists some linear polynomial $\ell = \sum_{i=1}^n c_i x_i \in I$ such that $S$ is exactly the set $\Supp(\ell) :=\{i \in [n] \mid c_i \neq 0 \}$~\cite[Lemma 4.1.4]{MaclaganSturmfels15}.
  The minimal dependent sets $\cC(I)$ are the circuits of $M(I)$, and for each $C \in \cC(I)$ there is a unique linear polynomial up to scaling $\ell_C$ such that $\Supp(\ell_C) = C$.
  Moreover, this set of linear polynomials completely determines $\Trop(I)$.

\begin{lemma}[see {\cite[Lemma 4.3.16]{MaclaganSturmfels15}}]
    Let $I \subseteq K[x^\pm]$ be a linear ideal.
    Write $\ell_C = \sum_{i \in C} \ell_{C,i} x_i$ for the linear polynomial corresponding to $C \in \cC(I)$.
    Then
      \[
      \Trop(I) = \bigcap_{C \in \cC(I)} \Trop(\ell_C) \, ,
      \]
    where each maximal cone has multiplicity one.
    In particular, $w \in \Trop(I)$ if and only if $\min_{i \in C}(w_i + \val(\ell_{C,i}))$ is attained at least twice for all $C \in \cC(I)$.
\end{lemma}

When $I \subseteq K[x^\pm]$ is a linear ideal,
we say that $\Trop (I)$ is a (realisable) \emph{tropical linear space}.

\begin{example}  \label{ex:tropicalLinearSpace}
  Consider the linear ideal
  \begin{equation*}
    I\coloneqq\Big\langle
    x_1+x_3+x_4,\;\;
    x_2+x_3+(1+t)x_4
    \Big\rangle\subseteq\CC\{\!\{t\}\!\}[x^\pm],
  \end{equation*}
Its collection of support-minimal linear polynomials is
\begin{align*}
\ell_{123} & = (1+t)x_1 - x_2 + tx_3 & \ell_{124} &= -x_1 + x_2 + tx_4 \\
\ell_{134} &= x_1+x_3+x_4 & \ell_{234} &= x_2+x_3+(1+t)x_4 \, ,
\end{align*}
hence $\Trop(I)$ is the intersection of tropical hypersurfaces $\Trop(\ell_C)$ for $C \in \cC(I) = \binom{[4]}{3}$.
This implies that $M(I)$ is the rank-2 uniform matroid with 4 elements.
See \cref{fig:tropicalLinearSpace} for an illustration.
\end{example}

\begin{figure}[t]
  \centering
  \begin{tikzpicture}
    \node[right] (linearSpace)
    {
      \begin{tikzpicture}[x={(220:12mm)},y={(0:12mm)},z={(90:20mm)}]
        \coordinate (v1) at (0,0,-0.25);
        \coordinate (v2) at (0,0,0.25);
        \draw[thick]
        (v1) -- (v2);
        \draw[->,blue,thick]
        (v2) -- ++(0,0.75,0.75) node[anchor=south,font=\scriptsize] {$e_1$}
        node[midway,anchor=north west,font=\scriptsize] {$\left\{w\in\RR^4 \; \big| \; w_1 \geq w_2 = w_3 + 1 = w_4 + 1  \right\}$};
        \draw[->,thick]
        (v2) -- ++(0,-0.75,0.75) node[anchor=south,font=\scriptsize] {$e_2$};
        \draw[->,thick]
        (v1) -- ++(0.75,0,-0.75) node[anchor=north,font=\scriptsize] {$e_3$};
        \draw[->,thick]
        (v1) -- ++(-0.75,0,-0.75) node[anchor=north,font=\scriptsize] {$e_4$};
        \fill
        (v1) circle (2pt)
        (v2) circle (2pt);
        \node[left,font=\scriptsize] at (v1) {$(0,0,0,0)$};
        \node[left,font=\scriptsize] at (v2) {$(1,1,0,0)$};
      \end{tikzpicture}
    };
  \end{tikzpicture}\vspace{-2mm}
  \caption{$\Trop(I)$ for $I$ in \cref{ex:tropicalLinearSpace}.
  We exploit the fact that $\Trop(I)$ is invariant under translation by $\RR\cdot (1,1,1,1)$ to produce pictures in $\RR^3$.}
  \label{fig:tropicalLinearSpace}
\end{figure}
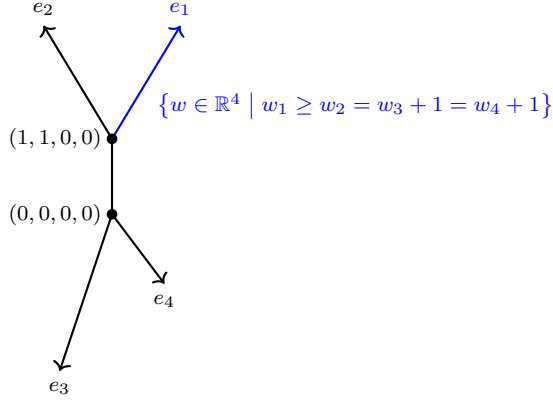

We will be particularly interested in the case where $I \subseteq K[x^\pm]$ is generated over a subfield $K' \subset K$ on which the valuation is trivial, e.g. $\CC \subset \CC\{\!\{t\}\!\}$.
In this case, $\val(\ell_{C,i}) \in \{0, \infty\}$ for all $C \in \cC(I)$, and hence the tropical linear space $\Trop(I)$ is completely determined from the underlying matroid $M(I)$.
The following lemma demonstrates that, in this case, the tropical linear space is precisely a Bergman fan up to a \emph{refinement} of the fan structure, i.e. every cone of the tropical linear space is a union of cones of the Bergman fan.

\begin{lemma}
  \label{lem:bergmanFan}
  Suppose $I\subseteq K[x^\pm]$ is a linear ideal generated over a subfield $K'\subseteq K$ on which the valuation $\val\colon K^\ast\rightarrow\RR$ is trivial. Then $\Trop(I)$ is equal to the Bergman fan $\Trop(M(I))$ up to a refinement of the fan structure.
\end{lemma}

\begin{proof}
    By \cite[Proposition 4.4.4] {MaclaganSturmfels15}, the complexes $\Trop(I)$ and $\Trop(M(I))$ have the same support. By \cite[Proposition 1]{ArdilaKlivans:2006}, it follows that $\Trop(M(I))$ is a refinement of $\Trop(I)$.
\end{proof}

As the fan structure of $\Trop(M(I))$ is a refinement of the fan structure of $\Trop(I)$, its maximal cones inherit multiplicity one from the maximal cones of $\Trop(I)$.
Moreover, it remains balanced by \cite[Lemma 3.6.2]{MaclaganSturmfels15}.
As such, we will freely swap between $\Trop(I)$ and $\Trop(M(I))$ for computations\footnote{\cite{MaclaganSturmfels15} defines a tropical variety to be the support of a balanced polyhedral complex precisely because the choice of fan structure does not matter, it suffices to know there is one.}.

Next we introduce the notions of balanced complexes and stable intersection. Whenever we intersect tropical varieties, we will generally mean stable intersections.

\begin{definition}\label{def:balanced}
Given a polyhedron $\eta \subseteq \RR^n$, we define the lattice $N_\eta := \ZZ^n \cap \Span(\eta - u)$ for some arbitrary element $u \in \eta$. If $\Sigma$ is a weighted polyhedral complex, denote by $\mult_\Sigma(\sigma)$ the multiplicity of a maximal cone $\sigma$ of $\Sigma$.

Let $\Sigma$ be a $d$-dimensional weighted polyhedral complex in $\RR^n$ that is \emph{pure}, in that every maximal cell has the same dimension. Fix a $(d-1)$-dimensional cell $\tau$ and, for each $d$-dimensional cell $\sigma$ such that $\sigma\supsetneq\tau$, let $v_\sigma\in\ZZ^n$ be a vector such that
$\ZZ\cdot v_\sigma + N_\tau = N_\sigma$.
We say that $\Sigma$ is \emph{balanced at} $\tau$ if $\sum_{\sigma\supsetneq\tau}\mult_\Sigma(\sigma)v_\sigma\in N_\tau$, and that $\Sigma$ is a \emph{balanced polyhedral complex} if it is balanced at every $(d-1)$-dimensional cell.
\end{definition}

\begin{definition}
  \label{def:stableIntersection}
  Let $\Sigma_1,\Sigma_2$ be two balanced polyhedral complexes in $\RR^n$ where the multiplicity of the maximal cell $\tau \in \Sigma_i$ is denoted $\mult_{\Sigma_i}(\tau)$.
  Their \emph{stable intersection} is the polyhedral complex consisting of the polyhedra
  \[
    \Sigma_1\wedge\Sigma_2 \coloneqq \Big\{\sigma_1\cap\sigma_2\bigmid\sigma_1\in\Sigma_1, \sigma_2\in\Sigma_2 \, , \, \dim(\sigma_1+\sigma_2)=n \Big\} \, .
  \]
  The multiplicity of the top-dimensional $\sigma_1\cap\sigma_2\in\Sigma_1\wedge\Sigma_2$ is given by
  \[
    \mult_{\Sigma_1\wedge\Sigma_2}(\sigma_1\cap\sigma_2) \coloneqq \sum_{\tau_1,\tau_2}\mult_{\Sigma_1}(\tau_1)\mult_{\Sigma_2}(\tau_2)[ \mathbb{Z}^n : N_{\tau_1}+N_{\tau_2}],
  \]
  where the sum is over all maximal cells $\tau_i\in\Sigma_i$ with $\sigma_1 \cap \sigma_2 \subseteq \tau_i$ for $i=1,2$, and $\tau_1\cap (\tau_2+\varepsilon\cdot v)\neq\emptyset$ for a fixed generic $v\in\RR^n$ and $\varepsilon>0$ sufficiently small.
  The integer $[\mathbb{Z}^n : N_{\tau_1}+N_{\tau_2}]$ denotes the \emph{index} of the sublattice $N_{\tau_1} + N_{\tau_2} \subseteq \mathbb{Z}^n$.
Since the complexes are balanced, the multiplicity of the stable intersection does not depend of the choice of $v$, hence it is well-defined. 
\end{definition} 

\begin{remark}
  \label{rem:stableIntersection}
  If $\Sigma_1,\Sigma_2$ are balanced polyhedral complexes, then the stable intersection $\Sigma_1\wedge\Sigma_2$ is either empty or a balanced polyhedral complex of codimension $\codim(\Sigma_1)+\codim(\Sigma_2)$ \cite[Theorem 3.6.10]{MaclaganSturmfels15}.  Moreover, we can also characterise the stable intersection as the limit of the generic perturbation
  \begin{equation} \label{eq:stable+intersection+perturbation}
    \Sigma_1\wedge\Sigma_2 = \lim_{\varepsilon\rightarrow 0} \Sigma_1\cap (\Sigma_2+\varepsilon\cdot v)
  \end{equation}
  for any generic $v\in\RR^n$, where the multiplicity of a point is the sum of the multiplicities of all points that tend to it \cite[Proposition 3.6.12]{MaclaganSturmfels15}.  The stable intersection is associative \cite[Remark 3.6.14]{MaclaganSturmfels15}, hence the stable intersection $\Sigma_1 \wedge \cdots \wedge \Sigma_k$ is well-defined for $k>2$.
\end{remark}

\begin{definition}
  \label{def:tropicalIntersectionProduct}
  Let $\Sigma_1,\dots,\Sigma_k$ be balanced polyhedral complexes in $\RR^n$ of complementary dimension, i.e., $\codim(\Sigma_1)+\dots+\codim(\Sigma_k)=n$.
  Their \emph{tropical intersection product} is the number of points in their stable intersection counted with multiplicity:
  \[ \Sigma_1\cdot\ldots\cdot \Sigma_k\coloneqq \sum_{p \in\Sigma_1 \wedge \cdots \wedge \Sigma_k} \mult_{\Sigma_1 \wedge \cdots \wedge \Sigma_k}(p) \, . \]
\end{definition}

We close the tropical preliminaries paper with the commonly known fact that the tropical intersection number of balanced complexes is invariant under translation.

\begin{lemma}
  \label{lem:tropicalIntersectionProductTranslationInvariant}
  Let $\Sigma_1,\dots,\Sigma_k$ be balanced polyhedral complexes in $\RR^n$ of complementary dimension and let $v_1,\dots,v_k\in\RR^n$.  Then
  \[ \Sigma_1\cdot\ldots\cdot \Sigma_k = (\Sigma_1+v_1)\cdot \ldots\cdot(\Sigma_k+v_k) \, . \]
\end{lemma}
\begin{proof}
  Without loss of generality, we may assume that $k=2$ and that $v_2=(0, \ldots,0)$. Consider the function $m : [0,1] \rightarrow \ZZ$ given by $t \mapsto (\Sigma_1+t\cdot v_1)\cdot \Sigma_2$.  
  By \cref{rem:stableIntersection}, the tropical intersection product is invariant under perturbation, hence $m$ is locally constant on $[0,1]$.  Since $[0,1]$ is connected, it follows that $m$ is constant.
\end{proof}

\subsection{Generic root count preliminaries}\label{subsect:genericRootCountPreliminaries}

In Section~\ref{sect:numberViaTropInt} we will use the basics and notation of parametrised polynomial systems, and specifically root counts, as a bridge between the realisation number and the tropical intersection product of some ideals.

\begin{definition}
  Let
  \begin{equation*}
      \CC[a][x^\pm]\coloneqq \CC\big[a_{j}\mid j\in [m] \big]\big[x_i^{\pm1}\mid i\in [n]\big]
  \end{equation*}
  be a \emph{parametrised (Laurent) polynomial ring} with parameters $a_j$ and variables $x_i$. Let $f\in \CC[a][x^\pm]$ be a parametrised polynomial, say $f=\sum_{\alpha\in\ZZ^n}c_\alpha x^\alpha$ with $c_\alpha\in \CC[a]$, and let $I\subseteq \CC[a][x^\pm]$ be a parametrised polynomial ideal.  We define their \emph{specialisation} at a choice of parameters $P\in \CC^m$ to be
  \[ f_P\coloneqq\sum_{\alpha\in\ZZ^n}c_\alpha(P)\cdot x^\alpha\in \CC[x^\pm] \quad\text{and}\quad I_P\coloneqq\langle h_P\mid h \in I \rangle\subseteq \CC[x^\pm]. \]
  Moreover, the \emph{root count} of $I$ at $P$ is defined to be the vector space dimension $\ell_{I,P}\coloneqq \dim_\CC (\CC[x^\pm]/I_P)\in\ZZ_{\geq 0}\cup\{\infty\}$.
\end{definition}

\begin{definition}
  Let $I\subseteq \CC[a][x^\pm]$ be a parametrised polynomial ideal.  Let $\CC(a)\coloneqq \CC(a_j\mid j\in [m])$ denote the rational function field in the parameters $a_j$.  The \emph{generic specialisation} of $I$ is the ideal in $\CC(a)[x^\pm]$ generated by $I$, i.e.,
  \[ I_{\CC(a)}\coloneqq \langle h\mid h \in I\rangle \subseteq \CC(a)[x^\pm]. \]
  The \emph{generic root count} of $I$ is $\ell_{I,\CC(a)}\coloneqq \dim_{\CC(a)}(\CC(a)[x^\pm]/I_{\CC(a)})\in\ZZ_{\geq 0}\cup\{\infty\}$.

  We say that $I$ is generically a complete intersection if $I_{\CC(a)}$ is a complete intersection, and that $I$ is generically zero-dimensional if $I_{\CC(a)}$ is zero-dimensional (in which case $\ell_{I,\CC(a)}<\infty$).
\end{definition}

\begin{remark}
  \label{rem:genericRootCount}
  The name ``root count'' for the vector space dimension $\ell_{I,P}$ is derived from the fact that it is the number of roots counted with a suitable algebraic multiplicity \cite[\S 4, Corollary 2.5]{CLS05}. The name ``generic root count'' for the vector space dimension $\ell_{I,\CC(a)}$ is justified by the fact that there is an Zariski-open subset in the parameter space $U\subseteq\CC^{|a|}$ over which it is attained, i.e., $\ell_{I,P}=\ell_{I,\CC(a)}$ for all $P\in U$.
\end{remark}

\begin{example}
    Consider the parametrised principal ideal
    \[ I\coloneqq\langle a_0+a_1x+a_2x^2\rangle\subseteq\CC[a_0,a_1,a_2][x^\pm]. \]
    Then $\ell_{I,(0,0,0)}=\infty$, $\ell_{I,(1,0,0)}=0$, $\ell_{I,(0,1,0)}=1$ and the generic root count is $\ell_{I,\CC(a)}=2$, which is attained whenever $a_2\neq 0$.
\end{example}

Let $I_1,\ldots,I_r$ be ideals of $\mathbb{C}[a][x^{\pm}]$.
We say that $I_1,\ldots,I_r$ are \emph{parametrically independent} if there exists a partition $\bigsqcup_{i=1}^r A_i = [m]$ and generator sets $F_i \subset \mathbb{C}[a_j : j \in A_i][x^{\pm}]$ such that $I_i = \langle F_i \rangle$ for each $i \in [r]$.

Recall that the algebraic torus $(\mathbb{C}^*)^n$ is an algebraic group with group multiplication
\begin{equation*}
    (\mathbb{C}^*)^n \times (\mathbb{C}^*)^n \rightarrow (\mathbb{C}^*)^n, ~ \big((s_i)_{i \in [n]}, (t_i)_{i \in [n]} \big) \mapsto (s_i \cdot t_i)_{i \in [n]}.
\end{equation*}
A parametrised polynomial ideal $I$ of $\mathbb{C}[a][x^{\pm}]$ is \emph{torus equivariant} if there exists a torus group action
\begin{equation*}
    (\mathbb{C}^*)^n \times \mathbb{C}^m \rightarrow \mathbb{C}^m, ~ \big(t,P\big) \mapsto t * P
\end{equation*}
such that $V(I_{t * P}) = t \cdot V(I_P)$, i.e. the natural torus group action on $V(I_P)$ corresponds to a torus group action on parameter space.

We now state our key tool for analysing parametric polynomial systems via tropical geometry.
Although this was first shown in~\cite{HR22}, we give a simplified version that appeared in \cite{holtren2023}.

\begin{proposition}[{\cite[Proposition 1]{holtren2023}}]\label{prop:keyprop}
    Let $I_1,\ldots, I_r$ be parametrised polynomial ideals of $\mathbb{C}[a][x^{\pm}]$ such that $\sum_{i=1}^r \codim I_{i, \mathbb{C}(a)} = n$.
    Further suppose that:
    \begin{enumerate}
        \item $I_1,\ldots,I_{r-1}$ are torus equivariant;
        \item $I_1,\ldots,I_{r}$ are parametrically independent.
    \end{enumerate}
    Then,
    given $I = I_1 + \ldots + I_r$,
    the following equality holds for any generic $P \in \mathbb{C}^m$:
    \begin{equation*}
        \ell_{I, \mathbb{C}(a)} = \Trop (I_{1,P}) \cdot \ldots \cdot \Trop (I_{r,P}).
    \end{equation*}
\end{proposition}


\section{Realisation numbers via tropical intersection theory}\label{sect:numberViaTropInt}

In this section we express the realisation number $c_2(G)$ as a tropical intersection product. We do this by tropicalising a zero-dimensional ideal whose generic root count relates to the realisation number.
We begin with any positive integer $d$.
As a consequence of \Cref{lem:maxwelllike}, we label the vertices so that $1i$ is an edge for all $2 \leq i \leq d$.

\begin{definition}
  \label{def:polynomialSystemEdgeVariables}
  Let $G$ be minimally $d$-rigid for some $d\in\ZZ_{>0}$.
  Consider the parametrised Laurent polynomial ring
  \begin{align*}
    \CC[a,c][y^\pm]&\coloneqq \CC \Big[a_{ij}\mid ij \in E(G), i<j \Big] \Big[ c_{l,k}\mid l\in[d-1], k\in[d]\Big]\\
                   &\qquad \left[ y_{ij,k}^{\pm1}\mid ij\in E(G), i<j, k\in [d] \right]
  \end{align*}
  with parameters $a_{ij}, c_{l,k}$ and variables $y_{ij,k}$.
  Let $I_E$ be the parametrised polynomial ideal generated by
  \begin{equation}
    \label{eq:generalSystemEdgeVariables}
    \begin{aligned}
      f_{ij}&\coloneqq \sum_{k=1}^d y_{ij,k}^2-a_{ij} &&\text{for }  ij \in E(G)\\
      g_{i,l}&\coloneqq \sum_{k=1}^d c_{l,k}y_{1i,k} &&\text{for } i\in[d]\setminus\{1\}\text{ and }l\in[d+1-i]\\
      h_{C,k}&\coloneqq \sum_{(s,t)\in C} y_{st,k} &&\text{for each directed cycle $C$ of $G$ and } k\in [d],
    \end{aligned}
  \end{equation}
  where a directed cycle is a directed path starting and ending at the same vertex with all other vertices distinct, and we again define $y_{st,k}\coloneqq -y_{ts,k}$ for $s>t$.
  We will refer to the $f_{ij}$ as the \emph{edge-length polynomials}, and $g_{i,l}$ as the \emph{vertex-pinning polynomials}.
\end{definition}

It is somewhat natural that these polynomials appear and count what we want. The variables $y_{ij,k}$ should be thought of as the coordinates of the vectors representing the edges of a realisation of $G$ in $\CC^d$ with vertex $1$ placed at the origin (or a fixed translate); hence, as their names suggest, the edge-length polynomials are exactly parametrising the generic edge-lengths, and the vertex-pinning polynomials are there to fix linear spaces where the vertices $\{2,\ldots,d\}$ are in order to eliminate multiple congruent realisations being counted. The appearance of the last type of polynomials are also not surprising: the sum of vectors making a directed cycle should add up to zero in Euclidean space. The interesting fact is that these polynomials are exactly what we need. We illustrate this with the following example.

\begin{example}
    Consider the minimally 2-rigid graph $K_4^-$, as shown in \Cref{fig:k4-e}. We have that $[d]\setminus\{1\}=[2]\setminus\{1\}=\{2\}$ and $[d+1-2]=[2+1-2]=\{1\}$, so there is only one vertex-pinning polynomial
    $$g_{2,1}= c_{1,1}y_{12,1}+c_{1,2}y_{12,2}.$$
    Intuitively, it fixes a slope for the edge $12\in E(K_4^-)$. Together with the polynomial $f_{12} = y_{12,1}^2+y_{12,2}^2-a_{12}$, we can think of the vertex $2$ being fixed (although in $\CC^d$ this can be false), which corresponds to counting up to rotations. The rest of the edge-length polynomials $f_{13},f_{14},f_{23},f_{34}$ fix the lengths of the edges $\{13,14,23,34\}$. The polynomials
    \begin{equation*}
        \begin{aligned}
        h_{C_1,1} &= y_{12,1} + y_{23,1} - y_{13,1}  &  h_{C_1,1} &= y_{12,2} + y_{23,2} - y_{13,2} \\
        h_{C_2,1} &= y_{13,1} + y_{34,1} - y_{14,1}  &  h_{C_2,1} &= y_{13,2} + y_{34,2} - y_{14,2} \\
        h_{C_3,1} &= y_{12,1} + y_{23,1} + y_{34,1} - y_{14,1}  &  h_{C_3,1} &= y_{12,2} + y_{23,2} + y_{34,2} - y_{14,2}
        \end{aligned}
    \end{equation*}
    are there to ensure that the vectors $(y_{ij,1}),(y_{ij,2})$ actually come from a realisation of $K_4^-$ in $\CC^2$.
\end{example}

The variables $y_{ij,k}$ in terms of the edges of $G$ will give us the interpretation as a tropical intersection product in terms of Bergman fans, but the realisation number is defined through vertex coordinates given by a realisation. In order to relate these two we need an intermediate ideal connecting them. The idea is simple: the coordinate $y_{ij,k}$ should represent the $k$-th coordinate $x_{i,k}-x_{j,k}$ (although this is the edge from $j$ to $i$ geometrically, we think of the edge from $i$ to $j$).

\begin{definition}
  \label{def:polynomialSystemVertexVariables}
  Given a positive integer $d$,
  let $G$ be minimally $d$-rigid with $n \geq d+1$.  Consider the parametrised Laurent polynomial ring
  \begin{equation*}
    \CC[a,b][x^\pm]\coloneqq \CC \Big[a_{ij}\mid ij \in E(G), i<j \Big] \Big[ b_{l,k}\mid l,k\in[d] \Big]\left[ x_{i,k}^{\pm1}\mid i\in [n], k\in [d] \right]
  \end{equation*}
  with parameters $a_{ij}, b_{l,k}$ and variables $x_{i,k}$.  Let $I_V\subseteq \CC[a,b][x^\pm]$ be the ideal generated by
  \begin{equation}
    \label{eq:generalSystemVertexVariables}
    \begin{aligned}
      f_{ij}^V&\coloneqq \sum_{k=1}^d (x_{i,k}-x_{j,k})^2 - a_{ij} &&\text{for } ij \in E(G), \mbox{ and}\\
      g_{i,l}^V&\coloneqq \sum_{k=1}^d b_{l,k}(x_{i,k}-1) &&\text{for } i\in[d] \text{ and } l\in [d+1-i].
    \end{aligned}
  \end{equation}
  We will refer to the $f_{ij}^V$ as the \emph{edge-length polynomials}, and to the $g_{i,l}^V$ as the \emph{vertex-pinning polynomials}.
\end{definition}

These new edge-length polynomials relate directly to the previous ones, both fixing edge-lengths. The vertex-pinning polynomials are the equations defining $X$ in \Cref{lem:realisationNumberAsGenericFibreCardinality}. In particular,
if $P$ is a generic choice of parameters, then $I_{V,P}$ is the ideal of the fibre of a generic point under the restricted rigidity map as defined in \Cref{lem:realisationNumberAsGenericFibreCardinality} and, by the same lemma, we should then have a generic root count equal to $2^dc_d(G)$. To prove it, however, we need to make this precise. The main detail is that the genericity condition for the $b_{l,k}$ in the definition of $X$ is in real affine space, while the genericity condition for the generic root count of $I_V$ is in complex affine space.  For this let us recall some basic results on algebraic geometry:
\begin{enumerate}[label=(\alph*)]
    \item\label{rem:1forLemma} If $U\subseteq\CC^r$ is a non-empty Zariski open subset of $\CC^r$, then $U\cap\RR^r$ is a non-empty Zariski open subset of $\RR^r$ (see, for example, \cite[Lemma 3.7]{dewar2024number}).
    \item\label{rem:2forLemma} Let $U\subseteq\CC^{r_1}\times\CC^{r_2}$ be a non-empty Zariski open subset of $\CC^{r_1+r_2}$ and $\pi_i:\CC^{r_1}\times\CC^{r_2}\longrightarrow\CC^{r_i}$ the respective projections. Then, for any $y_0\in\pi_2(U)\subseteq\CC^{r_2}$, the projection of the preimage $\pi_1(\pi_2^{-1}(y_0)\cap U)\subseteq \CC^{r_1}$ is a non-empty Zariski open subset of $\CC^{r_1}$.
    Indeed, if
    \begin{equation*}
        U=\CC^{r_1+r_2}\setminus V(f_1(x,y),\ldots,f_k(x,y)),
    \end{equation*}
    then
    \begin{equation*}
        \pi_1(\pi_2^{-1}(y_0)\cap U)=\CC^{r_1}\setminus V(f_1(x,y_0),\ldots,f_k(x,y_0)).
    \end{equation*}
\end{enumerate}

\begin{lemma}
 \label{lem:realisationNumberGeneralVertexVariables}
 Given $I_V\subseteq\CC[a,b][x^\pm]$ as defined above, the generic root count $\ell_{I_V,\CC(a,b)}$ is equal to $2^d c_d(G)$.
\end{lemma}

\begin{proof}
Let $\pi_b:\CC^{|a|+|b|}\longrightarrow\CC^{|b|}$ be the projection to the $b$-coordinates
and define the sets:
\begin{align*}
    U_0&\coloneqq \left\{P\in\CC^{|a|+|b|} \mid \ell_{I_V,\CC(a,b)}=\ell_{I_V,P} \right\}\subseteq \CC^{|a|+|b|},\\
    U&\coloneqq\pi_b(U_0)\subseteq \CC^{|b|},\\
    U_\RR&\coloneqq U\cap\RR^{|b|}\subseteq \RR^{|b|} \mbox{ and}\\
    U'&\coloneqq \left\{(b_1,\ldots,b_d)\in\RR^{d\times d}=\RR^{|b|} \mid \det([b_1\cdots b_d])\neq 0 \right\}\subseteq \RR^{|b|}.
\end{align*}
By \cref{rem:genericRootCount}, $U_0$ is a non-empty Zariski open subset of $\CC^{|a|+|b|}$. Hence $U$ is a non-empty Zariski open subset of $\CC^{|b|}$. Together with \ref{rem:1forLemma}, this implies that $U_\RR$ is a non-empty Zariski open subset of $\RR^{|b|}$. Therefore, as $U'$ is a non-empty Zariski open subset of $\RR^{|b|}$, the irreducibility of $\RR^{|b|}$ gives us an intersection point $\beta\in U'\cap U_\RR\subseteq U'\cap U$.

Now, if we denote by     $\pi_a:\CC^{|a|+|b|}\longrightarrow\CC^{|b|}$ the projection to the $a$-coordinates, by \ref{rem:2forLemma} we have that $\pi_a(\pi_b^{-1}(\beta)\cap U_0)$ is a non-empty Zariski open subset of $\CC^{|a|}$. On the other hand, by \Cref{lem:realisationNumberAsGenericFibreCardinality} we have that
$$\left\{\lambda\in\CC^{|a|} \mid 2^d c_{d}(G) = \#\left(f_{G,d}^{-1}(\lambda) \cap X \right) \right\}$$
is a non-empty Zariski open subset of $\CC^{|a|}$, where $X$ is defined as in \cref{lem:realisationNumberAsGenericFibreCardinality} with basis $b_1,\ldots,b_d\in\RR^d$ given by $\beta=(b_1,\ldots,b_d)\in\RR^{d\times d}=\RR^{|b|}$. Thus, the irreducibility of $\CC^{|a|}$ gives us an intersection point $\alpha$ in
\begin{equation*}
    \pi_a \left(\pi_b^{-1}(p_0)\cap U_0 \right)\cap\left\{\lambda\in\CC^{|a|} \mid 2^d c_{d}(G) = \#\left(f_{G,d}^{-1}(\lambda) \cap X \right) \right\}.
\end{equation*}
Then, taking $P=(\alpha, \beta)$, we observe that:
\begin{enumerate}
    \item $P\in U_0$, and thus $\ell_{I_V,\CC(a,b)}=\ell_{I_V,P}$;
    \item $2^d c_{d}(G) = \#\left(f_{G,d}^{-1}(\alpha) \cap X \right)$.
\end{enumerate}
Finally, note the equations $f_{ij,P}^V=0$ cut out the fibre $f_{G,d}^{-1}(\alpha)$, and the equations $g_{i,l,P}^V=0$ cut out the restricted domain $X$, so that $\ell_{I_V,P}=\#(f_{G,d}^{-1}(\alpha) \cap X )$. By the observations (1) and (2) above, we have the desired equality.
\end{proof}

Now that we have $\ell_{I_V,\CC(a,b)}=2^d c_d(G)$, we want to prove that $\ell_{I_V,\CC(a,b)}=\ell_{I_E,\CC(a,c)}$. This lemma is inspired by \cite[Lemma 2.16]{cggkls}. 

\begin{lemma}
  \label{lem:realisationNumberGeneralEdgeVariables}
  Let $I_E$ be as defined in \cref{def:polynomialSystemEdgeVariables}. Then the generic root count $\ell_{I_E,\CC(a,c)}$ is equal to $2^dc_d(G)$.
\end{lemma}

\begin{proof}
    Assume the existence of some non-empty Zariski open subset $U_V\subseteq\CC^{|a|+|b|}$ and a map $\rho:U_V\rightarrow \CC^{|a|+|c|}$ with a Zariski open image such that for all $P\in U_V$ there is an isomorphism $\CC[x^\pm]/I_{V,P}\cong \CC[y^\pm]/I_{E,\rho(P)}$.
    By \cref{rem:genericRootCount}, there is some open set $U \subseteq \CC^{|a|+|b|}$ over which the generic root count $\ell_{I_V,\CC(a,b)}$ is obtained.
    Hence $U' := U \cap U_V$ is a Zariski open dense subset which for all $P \in U'$ we have $\CC[x^\pm]/I_{V,P}\cong \CC[y^\pm]/I_{E,\rho(P)}$, and in particular $\ell_{I_E,\rho(P)} = \ell_{I_V,P}$.
    Combined with Lemma~\ref{lem:realisationNumberGeneralVertexVariables}, this proves the generic root count $\ell_{I_E,\CC(a,c)}$ is equal to $2^dc_d(G)$.
    The remainder of the proof is constructing such an open set $U_V$, the map $\rho$, and the isomorphism $\CC[x^\pm]/I_{V,P}\cong \CC[y^\pm]/I_{E,\rho(P)}$.

   Define $\rho:U_V\rightarrow \CC^{|a|+|c|}$ to be the projection
    \begin{equation*}
        \rho\big(\alpha,(\beta_1,\ldots,\beta_d)\big)=\big(\alpha,(\beta_1,\ldots,\beta_{d-1})\big)
    \end{equation*}
    restricted to the Zariski open dense subset
    \begin{equation*}
    U_V\coloneqq\Big\{\big((\alpha_{ij}),(\beta_{l})\big)\in\CC^{|a|+|b|} \bigmid \det([\beta_1\cdots\beta_d])\neq 0 \Big\},
    \end{equation*}
    where each $\beta_l\coloneqq (\beta_{l,k})_{k\in[d]}$ is a vector in $\CC^d$ and $[\beta_1\cdots \beta_d]\coloneqq (\beta_{l,k})_{l\in[d],k\in[d]}$ is the matrix in $\CC^{d\times d}$ with rows $\beta_1,\dots,\beta_d$.
    As $\rho$ is a projection map, it is necessarily open.
    We claim that $\CC[x^\pm]/I_{V,P}\cong \CC[y^\pm]/I_{E,\rho(P)}$ for all $P \in U_V$.

    For this, let $P=(\alpha,(\beta_1,\ldots,\beta_d))\in U_V$ and consider the $\CC$-algebra homomorphisms
    \begin{equation*}
    \begin{aligned}
        \varphi:\CC[x^\pm]&\longrightarrow \CC[y^\pm]/I_{E,\rho(P)} \\
        x_{i,k}&\longmapsto 1+\sum_{(s,t)\in \gamma_{i\rightarrow 1}} \overline y_{st,k}
    \end{aligned} \qquad \text{and} \qquad
    \begin{aligned}
        \psi:\CC[y^\pm]&\longrightarrow \CC[x^\pm]/I_{V,P} \\
        y_{ij,k}&\longmapsto \overline x_{i,k}-\overline x_{j,k},\vphantom{\sum_{(s,t)\in \gamma_{i\rightarrow 1}}}
    \end{aligned}
    \end{equation*}
    where $\overline{f} := f + I_{V,P}$ and $\overline{g} := g + I_{E,\rho(P)}$ for any $f \in \CC[x^\pm]$ and $g \in \CC[y^\pm]$, and
    \begin{equation*}
        \gamma_{i\rightarrow 1}=\{(s_0,s_1),(s_1,s_2),\dots,(s_{r-1},s_r) \mid \{s_j,s_{j+1}\}\in E(G), s_0=i, s_r=1\}
    \end{equation*}
    is a fixed arbitrary directed path in $G$ connecting vertex $i$ to vertex $1$. Note that $\varphi$ is well-defined due to the polynomials $h_{C,k,\rho(P)}\in I_{E,\rho(P)}$.
    To see this, if $\gamma_{i \to 1}$ and $\delta_{i \to 1}$ are directed paths from $i$ to $1$, then $D= \gamma_{i \to 1} \cup \delta_{1 \to i}$ is a closed walk.
    As any closed walk can be decomposed into a union of directed cycles $\{C_1, \dots, C_m\}$, we can write
    \begin{align*}
    \sum_{(s,t)\in \gamma_{i\rightarrow 1}} \overline y_{st,k} - \sum_{(s,t)\in \delta_{i\rightarrow 1}} \overline y_{st,k} &= \sum_{(s,t)\in D} \overline y_{st,k} = \sum_{j=1}^m \sum_{(s,t)\in C_j} \overline y_{st,k} = \sum_{j=1}^m h_{C_j,k} \in I_{E, \rho(P)} \, ,
    \end{align*}
    showing $\varphi$ is well-defined.
    We need to prove that these maps define homomorphisms
    \begin{equation*}
        \overline{\varphi}:
        \CC[x^\pm]/I_{V,P}\longrightarrow \CC[y^\pm]/I_{E,\rho(P)},
        \qquad \overline{\psi}:\CC[y^\pm]/I_{E,\rho(P)}\longrightarrow \CC[x^\pm]/I_{V,P}
    \end{equation*}
    which are inverse to each other.

    First, note that $x_{1,k}-1\in I_{V,P}$, so that $\overline{x}_{1,k}=1$ in $\CC[x^\pm]/I_{V,P}$, for all $k\in[d]$. Indeed, as $\det([\beta_1\cdots\beta_d])\neq 0$, then each $x_{1,k}-1$ can be expressed as a linear combination of the $g_{1,l,P}^V=\sum_{k=1}^d \beta_{l,k}(x_{1,k}-1)\in I_{V,P}$ for $l\in[d]$.

    Now, it is clear that $\psi(f_{ij,\rho(P)})=\overline{f}_{ij,P}^V=0$ and $\psi(h_{C,k,\rho(P)})=0$. Also, 
    $$\psi(g_{i,l,\rho(P)})=\sum_{k=1}^d \beta_{l,k}(\overline{x}_{1,k}-\overline{x}_{i,k})=-\sum_{k=1}^d \beta_{l,k}(\overline{x}_{i,k}-1)=0.$$
    Hence, $\psi(I_{E,\rho(P)}) = 0$ and so $\overline{\psi}$ is well defined.
    On the other hand, as we can choose $\gamma_{i\rightarrow1}=\{(i,1)\}$ for $i\in[d]\setminus\{1\}$, we have that for all $l\in[d+1-i]$
    $$\varphi(g_{i,l,P}^V)= \sum_{k=1}^d \beta_{l,k} \sum_{(s,t)\in \gamma_{i\rightarrow 1}} \overline y_{st,k} = \sum_{k=1}^d \beta_{l,k} \overline y_{i1,k}=-\overline{g}_{i,l,\varphi(P)}=0,$$
    and hence $g_{i,l \rho(P)} \in I_{E,\rho(P)}$.
    Also, for any $\{ij\}\in E(G)$ with $\gamma_{i\to j}:=\gamma_{i\to 1}\cup(-\gamma_{j\to 1})$, we have
      \begin{align}
      \begin{split}
      \label{eq:EdgeDifference}
    \varphi\big(\overline{x}_{i,k} - \overline{x}_{j,k}\big)
      &= \sum_{(s,t) \in \gamma_{i\to 1}}\overline{y}_{st,k} - \sum_{(s,t) \in \gamma_{j\to 1}}\overline{y}_{st,k} = \sum_{(s,t) \in \gamma_{i\to j}}\overline{y}_{st,k} \\
      &=  - \overline{y}_{ji,k} + \underbrace{\overline{y}_{ji,k} + \sum_{(s,t) \in \gamma_{i\to j}}\overline{y}_{st,k} }_{=0}
      = \overline{y}_{ij,k},
      \end{split}
    \end{align}
    from which we can deduce that $\varphi(f_{ij,P}^V)=\overline{f}_{ij,\rho(P)}$. Hence, $\varphi(I_{V,P})=0$ and $\overline{\varphi}$ is well defined.

    Finally, \cref{eq:EdgeDifference} implies that $\overline\varphi\circ\overline\psi(\overline{y}_{ij,k})=\overline{y}_{ij,k}$ and (combined with $\overline{x}_{1,k}=1$)
    \begin{align*}
        \overline\psi\circ\overline\varphi(\overline{x}_{i,k}) &= \overline\psi\circ\overline\varphi(\overline{x}_{1,k} + \overline{x}_{i,k} - \overline{x}_{1,k})
    = \overline\psi\Big(1+ \sum_{(s,t)\in \gamma_{i\rightarrow 1}} \overline{y}_{st,k}\Big)\\
      &= 1+\sum_{(s,t) \in \gamma_{i\to 1}}\big(\overline{x}_{s,k}-\overline{x}_{t,k}\big)
    = 1+\overline{x}_{i,k}-\overline{x}_{1,k} = \overline{x}_{i,k}. \qedhere
    \end{align*}
\end{proof}

In the case where $d=2$, we can simplify \cref{lem:realisationNumberGeneralEdgeVariables} to the following.

\begin{lemma}
  \label{lem:realisationNumberTwoDimensionalEdgeVariables}
  Let $G$ be a minimally $2$-rigid with $n \geq 3$ and let $\CC[a,c][y^\pm]$ be as in \cref{def:polynomialSystemEdgeVariables}.  Let $I_E'\subseteq \CC[a,c][y^\pm]$ be the ideal generated by
  \begin{equation}
    \label{eq:twoDimensionalSystemEdgeVariables}
    \begin{aligned}
      f_{ij}'&\coloneqq y_{ij,1}\cdot y_{ij,2} - a_{ij} &&\text{for } \{ij\} \in E(G) \\
      g_{12,1}'&\coloneqq c_{1,1}y_{12,1} + c_{1,2}y_{12,2}, \\
      h_{C,k}'&\coloneqq \sum_{(s,t)\in C} y_{st,k} &&\text{for each directed cycle $C$ of } G \text{ and } k\in [2].
    \end{aligned}
  \end{equation}
  Then the generic root count $\ell_{I_{E}',\CC(a,c)}$ is equal to $4 c_2(G)$.
\end{lemma}

\begin{proof}
    The $f_{ij}$ in System~\eqref{eq:generalSystemEdgeVariables} can be transformed into the $f_{ij}'$ in System~\eqref{eq:twoDimensionalSystemEdgeVariables} by the following change of variables
  \begin{equation*}
    y_{ij,1}\mapsto y_{ij,1}+\mathfrak{i}\cdot y_{ij,2},\qquad y_{ij,2}\mapsto y_{ij,1}-\mathfrak{i}\cdot y_{ij,2}.
  \end{equation*}
  The polynomial $g_{12,1}$ is mapped to $g_{12,1}'$ up to a suitable change of the parameters $c_{1,k}$. Moreover, the image of the $h_{C,k}$ polynomials generate the same ideal as the $h_{C,k}'$ polynomials.  Consequently, their generic root counts coincide.
\end{proof}

Next, the generic root count from \cref{lem:realisationNumberTwoDimensionalEdgeVariables} can be reformulated.

\begin{lemma}
  \label{lem:SIAGA2}
  Let $G$ be minimally $2$-rigid with $n \geq 3$ and consider the modified parametrised Laurent polynomial ring
  \begin{equation*}
    \CC\Big[a_{ij}\mid \{ij\} \in E(G), i<j \Big]\Big[c_{1,k}\mid k\in[2]\Big]\Big[y_{ij}^\pm\mid \{ij\} \in E(G), i<j \Big],
  \end{equation*}
  with parameters $a_{ij}, c_{1,k}$ and variables $y_{ij}$.  Let $I_E''$ be the parametrised polynomial ideal generated by:
  \begin{align*}
    \label{eq:twoDimensionalModification2}
    h_{C,1}'' &\coloneqq \sum_{(i,j)\in C}y_{ij} &&\text{for each directed cycle $C$ of $G$},\\
    h_{C,2}'' &\coloneqq \sum_{(i,j)\in C}a_{ij}y_{ij}^{-1} &&\text{for each directed cycle $C$ of $G$},\\
    g_{12}'' &\coloneqq c_{1,1}y_{12}^2+c_{1,2}.
  \end{align*}
  Then $\ell_{I_E'',\CC(a,c)}=4c_2(G)$.
\end{lemma}
\begin{proof}
    This follows straightforwardly from \cref{lem:realisationNumberTwoDimensionalEdgeVariables}.  The system is a simple reformulation of System~\eqref{eq:twoDimensionalSystemEdgeVariables}, replacing $y_{ij,2}$ by $a_{ij}y_{ij,1}^{-1}$.  The polynomials $h_{C,k}'$ then become $h_{C,k}''$. The polynomial $g_{12,1}'$ becomes $c_{1,1}y_{12}+c_{1,2}a_{12}y_{12}^{-1}$, which we can replace by $g_{12}''$ without changing the generic root count.
\end{proof}

We now deduce the equivalent restatement of \Cref{thm:main}, in which we are able to express the realisation number of a minimally 2-rigid graph $G$ in terms of a tropical intersection product involving the Bergman fan of the graphic matroid of $G$.

\begin{theorem}  \label{thm:twoDimensionalRealisationNumber}
\label{thm:main}
  The 2-realisation number of a minimally 2-rigid graph $G$ with $n \geq 3$ vertices is described by the tropical intersection product
  \[ 2 c_2(G) = (-\Trop(M_G))\cdot \Trop(M_G)\cdot \Trop(y_{12}-1). \]
\end{theorem}

\begin{proof}
  Consider the ideals from Lemma~\ref{lem:SIAGA2}:
  \begin{align*}
    I_1 &= \langle h_{C,1}''\mid C\subseteq E(G) \text{ directed cycle}\rangle, &
    I_2 &= \langle h_{C,2}''\mid C\subseteq E(G) \text{ directed cycle}\rangle, &
    I_3 &= \langle g_{12}'' \rangle \, .
  \end{align*}
  We first describe their corresponding tropical varieties.

  First consider $I_{1,P}$ for generic $P$.
  By \Cref{lem:bergmanFan}, $\Trop(I_{1,P})$ is equal to $\Trop(M(I_{1,P}))$ where $M(I_{1,P})$ is the matroid whose circuits are the supports of the minimal-support linear polynomials in $I_{1,P}$.
  By construction, these are exactly the cycles of $G$, and hence $M(I_1) = M_G$ (see \Cref{ex:graphic+matroid}).

  Next consider $I_{2,P}$ for generic $P$ and the monomial map $\phi$ that sends $y_{ij} \mapsto y_{ij}^{-1}$.
  Then $\phi(I_{2,P})$ is a linear ideal, and so $\Trop(\phi(I_{2,P}))$ is equal to $\Trop(M_G)$ by the same argument as $\Trop(I_{1,P})$.
  In the notation of~\cite{MaclaganSturmfels15}, the tropicalisation of $\phi$ is the linear map $\Trop(\phi) \colon \RR^n \rightarrow \RR^n$ that sends $v \mapsto -v$.
  As \cite[Corollary 3.6]{MaclaganSturmfels15} states $\Trop(\phi(I_{2,P})) = \Trop(\phi)\Trop(I_{2,P})$, we deduce that $\Trop(I_{2,P})$ is equal to $-\Trop(M_G)$.

  Finally, consider $I_{3,P}$ for generic $P$.
  By \Cref{def:tropical+hypersurface}, we deduce that $\Trop(I_{3,P})$ is the hyperplane defined by $\{y_{12} = 0\}$.
  Moreover, the Newton subdivision $\cN(g_{12}'')$ is the convex line segment between $0$ and $2\cdot\chi_{12}$, hence the unique cell of $\Trop(I_{3,P})$ has multiplicity two.
  We note here that, since $\val(c_{1,1})= \val(c_{1,2}) = 1$,
  the tropical hypersurfaces $\Trop (g_{12}'')$ and $\Trop (y_{12}^2 -1)$ are the same weighted polyhedral complex.

  We next show that $I_1,I_2,I_3$ satisfy the properties of \Cref{prop:keyprop} so that we can write the generic root count as a tropical intersection number.
  As $\codim(\Trop(M_G)) = n-2$ and tropicalisation preserves dimension (see, for example, \cite[Theorem 3.3.5]{MaclaganSturmfels15}), we deduce that $\codim(I_{i,\CC(a,c)}) = n-2$ for $i = 1,2$.
  Moreover, $\codim(I_{3,\CC(a,c)}) = 1$ and hence $\sum_{i=1}^3 \codim(I_{i,\CC(a,c)}) = 2n-3$, the dimension of the ambient space.
  For torus equivariance, we write a choice of parameters as $P = (P_{ij}, P_{1,1}, P_{1,2}) \in \CC^{2n-1}$ where $P_{ij}$ corresponds to a choice of $a_{ij}$, and $P_{1,i}$ to a choice of $c_{1,i}$.
  It is straightforward to check that $I_2$ and $I_3$ are torus equivariant under the respective torus actions $*_2$ and $*_3$:
  \begin{align*}
  *_2 &\colon (\CC^*)^{2n-3} \times \CC^{2n-1} \rightarrow \CC^{2n-1}, (t,P) \mapsto  ((t_{ij}P_{ij}), P_{1,1}, P_{1,2}) \\
  *_3 &\colon (\CC^*)^{2n-3} \times \CC^{2n-1} \rightarrow \CC^{2n-1},
  (t,P) \mapsto  ((P_{ij}), t_{12}^{-2}P_{1,1}, P_{1,2}).
  \end{align*}
  Finally, the ideals are clearly parametrically independent.
  
  By \cref{lem:SIAGA2} and \cref{prop:keyprop}, we have
  \begin{equation*}
      4 c_2(G) = \Trop(I_{1,P})\cdot\Trop(I_{2,P})\cdot \Trop(I_{3,P}) = \Trop(M_G)\cdot(-\Trop(M_G))\cdot \Trop(y_{12}^2-1)
  \end{equation*}
  for generic $P \in \CC^{2n-1}$.
  The statement then follows from the fact that $\Trop(y_{12}^2-1)$ and $\Trop(y_{12}-1)$ coincide set-theoretically, but the one cell of $\Trop(y_{12}^2-1)$ is of multiplicity $2$, whereas the cell of $\Trop(y_{12}-1)$ has multiplicity $1$.
\end{proof}


\section{Realisation numbers via matroid intersection theory}
\label{sec:realisations_via_matroid_intersection}

In this section, we derive a combinatorial characterisation of $c_2(G)$ via the tropical intersection product in \Cref{thm:main}.
This is given by enumerating so-called \emph{intersection trees} between pairs of flats satisfying certain conditions.

To define our combinatorial characterisation, we will utilise a number of techniques from \cite{AEP:23} for intersecting Bergman fans and their negatives, but adapted to our setup.
We will begin by recalling some results for general loopless matroids, then restrict to our specific case of graphic matroids of minimally 2-rigid graphs when required.

\subsection{Intersecting arboreal pairs of loopless matroids}

Let $M$ be a loopless matroid on ground set $E$.
Recall from \eqref{eq:bergman+cone} that the cones of the Bergman fan $\Trop(M)$ are $\sigma_\cF:=  \cone(\Chi_{F_1}, \dots, \Chi_{F_s}) + \RR \cdot \Chi_E$ where $\cF = (F_1, \dots, F_s) \in \Delta(M)$ is a proper chain of flats of $M$.
We can always extend the proper chain of flats $\cF$ to a non-proper chain $(F_0, F_1, \dots, F_{s+1})$ where $F_0 := \emptyset$ and $F_{s+1} := E$.
We define the \emph{reduced flats} of $\cF$ to be the sets $\widetilde{F}_i := F_i \setminus F_{i-1}$ where $1 \leq i \leq s+1$.
Note that the reduced flats $(\widetilde{F}_1, \dots, \widetilde{F}_s, \widetilde{F}_{s+1})$ always give a set partition of $E$ into $r(M)$ pieces: this observation will allow us to utilise the tools of \cite{AEP:23}.

Given the proper chain of flats $\cF = (F_1, \dots, F_s) \in \Delta(M)$, we define the linear space
\begin{equation*}
L(\cF) := \{ y \in \RR^E \mid y_i = y_j \: \forall i, j \in \widetilde{F}_k \, , \, 1 \leq k \leq s+1 \} \, .
\end{equation*}
It is immediate that $\sigma_\cF \subseteq L(\cF)$, and that in fact $L(\cF)$ is the linear span of $\sigma_\cF$.
We will first consider how these linear spaces $L(\cF)$ intersect as a precursor to describing how the cones $\sigma_\cF$ intersect.

Let $N$ be another loopless matroid on the same ground set $E$ and let $\cH = (H_1, \dots, H_t) \in \Delta(N)$ a proper chain of flats of $N$.
We define the \emph{intersection graph} $\Gamma_{\cF, \cH}$ to be the bipartite (multi)graph with vertex set indexed by reduced flats and edge set indexed by elements of $E$:
\begin{align*}
V(\Gamma_{\cF, \cH}) &:= \{\widetilde{F}_1, \dots, \widetilde{F}_s, \widetilde{F}_{s+1}\} \sqcup \{\widetilde{H}_1, \dots, \widetilde{H}_t, \widetilde{H}_{t+1}\} \, , \\
E(\Gamma_{\cF, \cH}) &:= \{e(a) := (\widetilde{F}_i, \widetilde{H}_j) \mid a \in \widetilde{F}_i \cap \widetilde{H}_j \} \, .
\end{align*}
As the reduced flats of $\cF$ and $\cH$ give partitions of $E$, there is a one-to-one correspondence between $E$ and the edges of $\Gamma_{\cF, \cH}$.
We write $e(a)$ for the edge of $E(\Gamma_{\cF, \cH})$ labelled by $a$.

We would like to exploit the combinatorics of $\Gamma_{\cF, \cH}$ to determine how the affine spaces $L(\cF) \cap (\alpha - L(\cH))$ intersect for various $\alpha \in \RR^E$.
However, we first note that both of these linear spaces contain the span of the all-ones vector $\Chi_E$.
Hence, if $z$ is contained in their intersection then the line $z + \RR\cdot \Chi_E$ is also contained. 
To cut out this redundancy, we will write $L(\cF)|_{y_\epsilon = 0}$ for the linear space $L(\cF)$ with the additional affine constraint that $y_\epsilon = 0$ for some fixed $\epsilon \in E$.
We will use the same notation when considering cones and Bergman fans with an additional affine constraint.

\begin{lemma}[{\cite[Lemma 2.2]{AEP:23}}] \label{lem:intersection+graph}
    Let $M$ and $N$ be loopless matroids with proper chains of flats $\cF \in \Delta(M)$ and $\cH \in \Delta(N)$.
    Let $\Gamma_{\cF, \cH}$ be their intersection graph. Then the following hold.
    \begin{enumerate}
        \item If $\Gamma_{\cF, \cH}$ has a cycle, then $L(\cF)|_{y_\epsilon = 0} \cap (\alpha - L(\cH)) = \emptyset$ for generic $\alpha \in \RR^E$.
        \item If $\Gamma_{\cF, \cH}$ is disconnected, then $L(\cF)|_{y_\epsilon = 0} \cap (\alpha - L(\cH))$ is not a point for any $\alpha \in \RR^E$.
        \item If $\Gamma_{\cF, \cH}$ is a tree, then $L(\cF)|_{y_\epsilon = 0} \cap (\alpha - L(\cH))$ is a point for any $\alpha \in \RR^E$.
    \end{enumerate}
\end{lemma}

\begin{example}
We will primarily be interested in the intersection graphs arising from flags of flats of graphic matroids.
\Cref{fig:intersection_graph} shows three different graphs $G = (V,E)$ where $|V| = 4$ and their intersection graphs $\Gamma_{\cF, \cH}$ for the chains of flats
\begin{align} \label{eq:chain+of+flats}
\begin{split}
\cF \colon \cl(1) \subsetneq \cl(1,2) \subsetneq \cl(1,2,3) \in \Delta(M_G) \, , \\
\cH \colon \cl(4) \subsetneq \cl(4,3) \subsetneq \cl(4,3,2) \in \Delta(M_G) \, .
\end{split}
\end{align}
As each graph has a spanning tree of size three, the resulting linear spaces $L(\cF)$ and $L(\cH)$ are both 3-dimensional in each example, but the dimension of their ambient space changes as the number of edges of $G$ changes.
\begin{figure}[t]
\includegraphics[width=\textwidth]{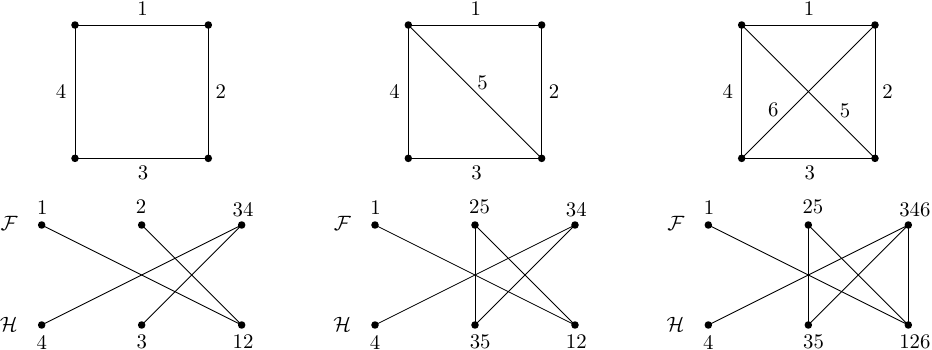}
\caption{A collection of graphs $G$ and their corresponding intersection graphs $\Gamma_{\cF,\cH}$ for the chains of flats $\cF, \cH \in \Delta(M_G)$ defined in \eqref{eq:chain+of+flats}.}
\label{fig:intersection_graph}
\end{figure}

In the first graph, where $G \cong C_4$ is the 4-cycle, the resulting intersection graph is disconnected.
By \Cref{lem:intersection+graph}, the linear spaces $L(\cF)|_{y_\epsilon = 0}$ and $(\alpha - L(\cH))$ do not intersect in a unique point for any generic $\alpha$.
A dimension check shows that $L(\cF)$ and $L(\cH)$ are codimension one subspaces of $\RR^4$, hence $L(\cF)|_{y_\epsilon = 0}$ and $(\alpha - L(\cH))$ generically intersect in a 1-dimensional subspace.

In the second graph, where $G \cong K_4^-$, the resulting intersection graph is a tree.
By \Cref{lem:intersection+graph}, the linear spaces $L(\cF)|_{y_\epsilon = 0}$ and $(\alpha - L(\cH))$ intersect in a unique point for any generic $\alpha$.
A dimension check shows that $L(\cF)$ and $L(\cH)$ are codimension two subspaces of $\RR^5$, hence $L(\cF)|_{y_\epsilon = 0}$ and $(\alpha - L(\cH))$ generically intersect in a 0-dimensional subspace.

In the third graph, where $G \cong K_4$, the resulting intersection graph is spanning but contains a cycle.
By \Cref{lem:intersection+graph}, the linear spaces $L(\cF)|_{y_\epsilon = 0} \cap (\alpha - L(\cH))$ have empty intersection for all generic $\alpha$.
A dimension check shows that $L(\cF)$ and $L(\cH)$ are codimension three linear subspaces of $\RR^6$, hence $L(\cF)|_{y_\epsilon = 0}$ and $(\alpha - L(\cH))$ generically do not intersect.
\end{example}

When $\Gamma_{\cF, \cH}$ is a tree, we call $(\cF, \cH) \in \Delta(M) \times \Delta(N)$ an \emph{arboreal pair}.
In this case, we can go further and use $\Gamma_{\cF, \cH}$ to read off the unique point of intersection of $L(\cF)|_{y_\epsilon = 0} \cap (\alpha - L(\cH))$.
As $y_i = y_j$ when $i, j \in \widetilde{F}_k$ for all $y \in L(\cF)$, we will label the coordinates by their partition part $y_{\widetilde{F}_k}$ for ease, which we do analogously with each $z_{\widetilde{H}_k}$ for $z \in L(\cH)$.
Finally, we write $\widehat{F}$ for the reduced flat of $\cF$ containing the distinguished element $\epsilon \in E$.

\begin{lemma}[{\cite[Lemma 2.3]{AEP:23}}] \label{lem:path}
    Let $(\cF, \cH)$ be an arboreal pair, $\Gamma_{\cF, \cH}$ be their intersection graph and $\alpha \in \RR^E$.
    The unique vectors $y \in L(\cF)$ and $z \in L(\cH)$ such that $y + z = \alpha$ and $y_\epsilon = 0$ are given by
    \begin{equation}\label{eq:alt+sum}
        y_{\widetilde{F}_i} = \alpha_{a_1} - \alpha_{a_2} + \cdots + (-1)^{k+1} \alpha_{a_k}
        \quad \text{and} \quad
        z_{\widetilde{H}_j} = \alpha_{b_1} - \alpha_{b_2} + \cdots + (-1)^{\ell+1} \alpha_{b_\ell}
    \end{equation}
    where $e(a_1)e(a_2)\dots e(a_k)$ is the unique path from $\widetilde{F}_i$ to $\widehat F$ and $e(b_1)e(b_2)\dots e(b_\ell)$ is the unique path from $\widetilde{H}_j$ to $\widehat F$
    for any $i$ and $j$.
\end{lemma}

While \Cref{lem:path} works for any vector $\alpha \in \RR^E$, we will have better combinatorial tools if we restrict to certain (generic) families of vectors.
Identify $E$ with the set $\{1, \dots, m\}$ where $|E| = m$.
We call $\alpha \in \RR^E$ \emph{rapidly increasing} if $\alpha_{i+1} > 3 \alpha_i > 0$ for all $1 \leq i \leq m-1$.
It is easy to verify that if $\alpha$ is rapidly increasing, then $\sum_{j=1}^i \delta_j \alpha_j < \alpha_{i+1} - \sum_{j=1}^i \varepsilon_j \alpha_j$ for all choices of signs $\delta_j, \varepsilon_j \in \{+1, -1\}$.
Observe that this naturally fixes a total order on the set $E$, and hence a total ordering on the edges of any intersection graph $\Gamma_{\cF, \cH}$.

If $\cF, \cH$ are an arboreal pair, each pair of vertices in $\Gamma_{\cF,\cH}$ have a unique path between them that alternates between vertices of $\cF$ and $\cH$.
We write $\widetilde{F}_{i} \rightarrow \widetilde{H}_j$ for the unique (oriented) path from $\widetilde{F}_{i}$ to $\widetilde{H}_j$.
We say that a path is \emph{$\cF$-maximal} if its largest edge $e(a) = (\widetilde{F}_{i}, \widetilde{H}_j)$ with respect to the total order is traversed from $\widetilde{F}_{i}$ to $\widetilde{H}_j$, and \emph{$\cH$-maximal} otherwise.

\begin{proposition}\label{prop:flat+cone+intersect}
    Let $M, N$ be two loopless matroids on $E$ such that $r(M) + r(N) = |E| + 1$ with respective proper chains of flats
    \[
    \cF = (F_1, \dots, F_s) \in \Delta(M) \, , \quad \cH = (H_1, \dots , H_t) \in \Delta(N) \, .
    \]
    Fix an edge $\epsilon \in E$ and choose $\alpha \in\RR^E$ to be generic and rapidly increasing.
    Then $(\sigma_\cF|_{y_{\epsilon} = 0}) \cap (\alpha - \sigma_\cH)$ is non-empty if and only if
    \begin{enumerate}
        \item $\Gamma_{\cF,\cH}$ is a tree,
        \item every path $\widetilde{F}_i \rightarrow \widetilde{F}_j$ is $\cF$-maximal for all $1 \leq i < j \leq s+1$, and
        \item every path $\widetilde{H}_i \rightarrow \widetilde{H}_j$ is $\cH$-maximal for all $1 \leq i < j \leq t+1$.
    \end{enumerate}
    Moreover, if the intersection is non-empty then it contains a single point, and $\cF$ and $\cH$ are maximal.
\end{proposition}

\begin{proof}
    Observe that the intersection graph $\Gamma_{\cF, \cH}$ has $s+t+2$ vertices and $|E|$ edges,
    where $s+ 1 \leq r(M)$ and $t +1 \leq r(N)$.
    As $r(M) + r(N) = |E| + 1$, it follows that $\Gamma_{\cF, \cH}$ has at most $|E| + 1$ vertices with equality if and only if $\cF$ and $\cH$ are both maximal.
    If $\Gamma_{\cF, \cH}$ has less that $|E|+1$ vertices, it must contain a cycle and hence $L(\cF)|_{y_\epsilon = 0} \cap (\alpha - L(\cH)) = \emptyset$ by \Cref{lem:intersection+graph}.
    If $\Gamma_{\cF, \cH}$ has exactly $|E|+1$ vertices, it either contains a cycle or is a tree.
    By \Cref{lem:intersection+graph}, the intersection $L(\cF)|_{y_\epsilon = 0} \cap (\alpha - L(\cH))$ is empty in the former case and a single point in the latter.
    This shows that $\Gamma_{\cF, \cH}$ must be a tree, and thus $\cF$ and $\cH$ must be maximal, to give a non-empty intersection.

    It remains to show that the unique points $y \in L(\cF)|_{y_\epsilon = 0}$ and $z \in L(\cH)$ such that $y + z = \alpha$ are contained in $\sigma_\cF|_{y_\epsilon = 0}$ and $\sigma_\cH$ respectively.
    This is equivalent to requiring $y_{\tF_i} - y_{\tF_j} \geq 0$ for all $1 \leq i < j \leq s+1$ and $z_{\tH_i} - z_{\tH_j} \geq 0$ for all $1 \leq i < j \leq t+1$.
    We now show that $y_{\tF_i} - y_{\tF_j} \geq 0$ is equivalent to the path $\tF_i \rightarrow \tF_j$ being $\cF$-maximal.
    The proof that $z_{\tH_i} - z_{\tH_j} \geq 0$ is equivalent to $\tH_i \rightarrow \tH_j$ is $\cH$-maximal is identical.

    We first show that $y_{\tF_i} - y_{\tF_j}$ is equal to the alternating sum
    \begin{align}\label{eq:alt+path}
    y_{\tF_i} - y_{\tF_j} = \alpha_{a_1} - \alpha_{a_2} + \cdots - \alpha_{a_k} \, ,
    \end{align}
    where $e(a_1), \dots, e(a_k)$ is the unique path $\widetilde{F}_i \rightarrow \widetilde{F}_j$ in $\Gamma_{\cF,\cH}$.
    To see this, let $\widehat F$ be the reduced flat containing $\epsilon$, let $e(b_1), \dots, e(b_\ell)$ be the path from $\tF_i$ to $\widehat F$ and $e(c_1), \dots, e(c_n)$ be the path from $\tF_j$ to $\widehat F$.
    As $\Gamma_{\cF,\cH}$ contains no cycles, there exists some $w \in \ZZ_{\geq 0}$ such that the last $w$ edges in these paths are equal, i.e. $e(b_{\ell - w' +1}) = e(c_{n - w' + 1})$ align for all $0 \leq w' \leq w$, and overlap nowhere else.
    Moreover, these paths to $\widehat{F}$ have the same parity, and so the last $w$ terms in \eqref{eq:alt+sum} have the same signs.
    As such, we have from \Cref{lem:path}
    \begin{align*}
    y_{\tF_i} - y_{\tF_j} &= (\alpha_{b_1} - \alpha_{b_2} + \cdots \pm \alpha_{b_\ell}) - (\alpha_{c_1} - \alpha_{c_2} + \cdots \pm \alpha_{c_n}) \\
    &= (\alpha_{b_1} - \alpha_{b_2} + \cdots \pm \alpha_{b_{\ell - w}}) - (\alpha_{c_1} - \alpha_{c_2} + \cdots \pm \alpha_{c_{n-w}}) \\
    &= \alpha_{b_1} - \alpha_{b_2} + \cdots \pm \alpha_{b_{\ell - w}} \mp \alpha_{c_{n-w}} \pm \cdots +\alpha_{c_2} - \alpha_{c_1} \, ,
    \end{align*}
    where $e(b_1), \dots,e(b_{\ell-w}),e(c_{n-w}), \dots, e(c_1)$ is a path from $\tF_i$ to $\tF_j$.
    As $\Gamma_{\cF, \cH}$ a tree, this is the unique path and hence equal to $e(a_1), \dots, e(a_k)$.

    As $\alpha$ is rapidly increasing, \eqref{eq:alt+path} is positive if and only if the largest $\alpha_{a_i}$ has a positive sign, which occurs if and only if the largest edge $e(a_i)$ is traversed from $\cF$ to $\cH$.
    This is precisely the condition of being $\cF$-maximal.
\end{proof}

When $\Gamma_{\cF,\cH}$ satisfies the three properties of \Cref{prop:flat+cone+intersect}, we call $(\cF, \cH) \in \Delta(M) \times \Delta(N)$ an \emph{intersecting arboreal pair} of $M$ and $N$ (with respect to $\alpha$).
We introduce this terminology as \Cref{prop:flat+cone+intersect} shows intersecting arboreal pairs combinatorially characterise the cones in which the fans $\Trop(M)|_{y_{\epsilon} = 0}$ and $\alpha - \Trop(N)$ intersect.

Observe that intersecting arboreal pairs are heavily dependent on the choice of $\alpha$ and the ordering it induces on $E$.
However, we will only be interested in the number of intersecting arboreal pairs for any choice of rapidly increasing generic $\alpha$.
The next lemma shows that this number does not depend on the choice of $\alpha$.

\begin{proposition}\label{prop:intersecting+arboreal+pair+invariant}
Let $M, N$ be loopless matroids on $E$ such that $r(M) + r(N) = |E| + 1$, let $\alpha \in \RR^E$ be generic and rapidly increasing, and let $\epsilon \in E$.
The number of intersecting arboreal pairs of $M$ and $N$ is equal to the tropical intersection product
\[
(-\Trop(N)) \cdot \Trop(M) \cdot \Trop(y_\epsilon - 1) \, .
\]
In particular, the number of intersecting arboreal pairs of $M$ and $N$ is independent of the choice of generic and rapidly increasing $\alpha \in \RR^E$.
\end{proposition}

\begin{proof}
Let us recall the notation $X|_{y_{\epsilon} = 0} := X \cap \{y \mid y_\epsilon = 0\}$ for any subset $X \subseteq \RR^E$.
We note that every cell of $\Trop(M)$ contains $\RR\cdot \Chi_E$ as a lineality space, and hence $\Trop(M)$ meets transversely with $\Trop(y_{\epsilon} - 1) = \{y \mid y_{\epsilon} = 0\}$.
 It follows from the definition of stable intersection that
    \[
    \Trop(M) \wedge \Trop(y_\epsilon - 1) = \Trop(M) \cap \{y \mid y_{\epsilon} = 0\} = \Trop(M)|_{y_{\epsilon} = 0} \, .
    \]
    Using this, we observe that for any generic $\alpha \in \RR^E$
    \begin{align} \label{eq:intersection+product}
(-\Trop(N)) \cdot \Trop(M) \cdot \Trop(y_{\epsilon} - 1) &= (-\Trop(N)) \cdot (\Trop(M)|_{y_{\epsilon} = 0}) \\
    &= \#\left((\Trop(M)|_{y_{\epsilon} = 0}) \cap (\alpha -\Trop(N)) \right) \, . \nonumber
    \end{align}
    The last equality follows by combining \eqref{eq:stable+intersection+perturbation} with \Cref{lem:tropicalIntersectionProductTranslationInvariant}, where the tropical intersection product is equal to the number of points (counted with multiplicity) in the intersection after any generic perturbation, namely $\alpha \in \RR^E$. 
    Moreover, the multiplicity of any point in $(\Trop(M)|_{y_{\epsilon} = 0}) \cap (\alpha -\Trop(N))$ is 1, as they are transversal intersections of cells from tropical linear spaces.
    Hence we are just counting the number of points.
    As this holds for any generic $\alpha$, we can choose $\alpha$ to be rapidly increasing.

    From the definition of the Bergman fan, the points in $(\Trop(M)|_{y_{\epsilon} = 0}) \cap (\alpha -\Trop(N))$ can be enumerated by counting the intersection points of $\sigma_\cF \cap (\alpha - \sigma_\cH)$ where $\cF \in \Delta(M), \cH \in \Delta(\cH)$ are proper chains of flats of $M$ and $N$ respectively.
    As such, we can apply \Cref{prop:flat+cone+intersect} to deduce that
    \begin{align*}
& \#\left((\Trop(M)|_{y_{\epsilon} = 0}) \cap (\alpha -\Trop(N)) \right) \\
=& \#\{(\cF,\cH) \in \Delta(M) \times \Delta(N) \mid (\cF,\cH) \text{ intersecting arboreal pair}\} \, .
    \end{align*}
    Coupling with \eqref{eq:intersection+product} gives the result.
\end{proof}

\subsection{Intersecting arboreal pairs for minimally 2-rigid graphs}

We now apply this machinery to our setting of realisation numbers of minimally 2-rigid graphs.
Combining this with \Cref{thm:main} gives the following combinatorial characterisation of the realisation number as a fairly immediate corollary.

\begin{theorem}\label{thm:arboreal}
    Let $G$ be a minimally 2-rigid graph and let $M_G$ be its graphic matroid.
    Then the realisation number $c_2(G)$ is equal to the number of distinct pairs of chains of flats of $M_G$ that form an intersecting arboreal pair, i.e.
    \[
    c_2(G) = \#\left\{\{\cF, \cH\} \in \binom{\Delta(M_G)}{2} \: \bigg| \: (\cF, \cH) \text{ intersecting arboreal pair} \right\}.
    \]
\end{theorem}

\begin{proof}
    Let $(G = [n], E)$ be a minimally 2-rigid graph with edge $e \in E$.
    Note that $M_G$ has rank $n-1$ on $|E| = 2n-3$ elements, hence $2r(M_G) = |E| +1$.
    As such, we can apply \Cref{prop:intersecting+arboreal+pair+invariant} in the case where $M = N = M_G$, along with \Cref{thm:main}, to deduce that
    \begin{align*}
    c_2(G) &= \frac{1}{2}(-\Trop(M_G)) \cdot \Trop(M_G) \cdot \Trop(y_e - 1) \\
    &= \frac{1}{2}\#\big\{(\cF,\cH) \in \Delta(M_G) \times \Delta(M_G) \mid (\cF,\cH) \text{ intersecting arboreal pair} \big\} \, .
    \end{align*}
    The final observation is that $(\cF, \cH)$ is an intersecting arboreal pair if and only if $(\cH, \cF)$ is also, as the conditions of \Cref{prop:flat+cone+intersect} are symmetric when $M = N$.
    Moreover, it also follows from \Cref{prop:flat+cone+intersect} that we can never have $(\cF,\cF)$ as an arboreal pair: the edges of $\Gamma_{\cF,\cF}$ are precisely $|\tF_i|$ copies of the edge $(\tF_i, \tF_i)$ for each vertex $\tF_i$, and so $\Gamma_{\cF,\cF}$ is not a tree.
    As such, we can restrict $\Delta(M_G)$ to $\binom{\Delta(M_G)}{2}$ to avoid double counting.
\end{proof}

\begin{example}\label{example: intersecting arboreal pairs}
Consider the graph $G = K_4^-$ with edges $[5]$ shown on the left of \Cref{fig:intersecting_arboreal_pairs}. We order the edges $1 > 2 > 3 > 4 > 5$. The rank-one flats are $1,2,3,4,5$, and the rank-two flats are $125, 13, 14, 23, 24, 345$. So, there are $14 = |\Delta(M_G)|$ maximal chains of flats in $M_G$. Consider the two pairs of chains of flats given by
\[
\{\cF_1 : 1 \subseteq 13,\,
 \cH_1 : 2 \subseteq 24\}
\quad \text{and} \quad
\{\cF_2 : 1 \subseteq 14,\,
 \cH_2 : 2 \subseteq 23\}.
\]
Their graphs, $\Gamma_{\cF_1, \cH_1}$ and $\Gamma_{\cF_2, \cH_2}$, are shown in the middle and right of \Cref{fig:intersecting_arboreal_pairs} respectively. Let us consider the graph $\Gamma_{\cF_1, \cH_1}$, which has vertices $\widetilde F_1$, $\widetilde F_2$, $\widetilde F_3$, labelled with $1$, $3$, $245$, and vertices $\widetilde H_1$, $\widetilde H_2$, $\widetilde H_3$, labelled with $2$, $4$, $135$ respectively. Observe that the directed path $\widetilde F_1 \rightarrow \widetilde F_2$ is given by the sequence of edges: $1 = (\widetilde F_1, \widetilde H_3)$, $3 = (\widetilde H_3, \widetilde F_2)$. The maximal edge in this path, with respect to our fixed ordering, is edge $1$. This edge is directed from $\widetilde F_1$ to $\widetilde H_3$, hence this path is $\cF$-maximal. Similarly, the maximal edge in each directed path $\widetilde F_1 \rightarrow \widetilde F_3$ and $\widetilde F_2 \rightarrow \widetilde F_3$ is directed from $\widetilde F_i$ to $\widetilde H_j$ for some $i$ and $j$, so they are both $\cF$-maximal. Moreover, each directed path $\widetilde H_u \rightarrow \widetilde H_v$ is $\cH$-maximal. So, we have shown that $\{\cF_1, \cH_1\}$ is an intersecting arboreal pair. Similarly, $\{\cF_2, \cH_2\}$ is another intersecting arboreal pair.

It turns out that these are the only two intersecting arboreal pairs, so by \Cref{thm:arboreal}, the $2$-realisation number of $G$ is $2$.
\end{example}

\begin{figure}[t]
    \centering
    \includegraphics[width=0.9\textwidth]{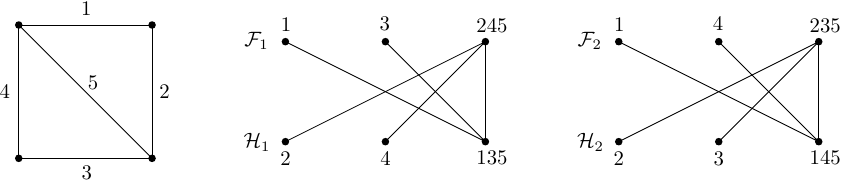}
    \caption{Graph $G$ (left), and graphs $\Gamma_{\cF_1, \cH_1}$ (middle) and $\Gamma_{\cF_2, \cH_2}$ (right) of the intersecting arboreal pairs $\{\cF_1, \cH_1\}$ and $\{\cF_2, \cH_2\}$ in \Cref{example: intersecting arboreal pairs}.}
    \label{fig:intersecting_arboreal_pairs}
\end{figure}

\section{Upper and lower bounds}\label{sec: upper and lower bounds}

We now utilise the combinatorial characterisation of the previous section (\Cref{thm:arboreal}) to provide both upper and lower bounds for $c_2(G)$.

\subsection{An upper bound on \texorpdfstring{$c_2(G)$}{c2(G)}}

While \Cref{thm:arboreal} gives a purely combinatorial condition for the realisation number, it's not clear how to compute it beyond brute force.
We give a candidate for bounding the realisation number via a known matroid invariant that is more amenable to computation.

\begin{definition}\label{def:nbc-basis}
Let $M$ be a matroid on ground set $E$ with some fixed total order $\prec$ on $E$.
A \emph{broken circuit} of $M$ is any set $\widetilde{C}:= C \setminus \min(i : i \in C)$, where $C$ is a circuit of $M$.
A \emph{non-broken circuit (nbc)-basis} (with respect to $\prec$) is a basis $B$ of $M$ that contains no broken circuits.
\end{definition}

It is a non-trivial fact that the number of nbc-bases does not depend on the total order (see, for example, \cite[Theorem 7.4.6]{Bjorner:1992}). We write $\nbc M$ for the number of nbc-bases of a matroid $M$ with respect to some total order.

One necessary condition for a basis $B$ to be an nbc-basis $B$ is that it must contain the minimal element $\min(E)$ of $E$.
Otherwise, if $B$ does not contain $\min(E)$, then there must exist some (fundamental) circuit $C$ such that $\min(E) \in C \subseteq B \cup \min(E)$, contradicting the assumption that $B$ contains no broken circuits.

We may characterise nbc-bases in terms of maximal chains of flats.
Let $B = \{b_1, \dots, b_k\}$ be a basis with $b_1 \succ \cdots \succ b_k$. For each $i \in [k]$, let $F_i = \cl(b_1, b_2, \dots, b_i) \subseteq E$ be a flat and write $F_0 = \cl(\emptyset) = \emptyset$.
With this it follows that $\rank F_i = i$ for each $i =0,1,\ldots,k$.
We define the maximal chain of flats associated to the basis $B$ as
\[
\cF(B) : F_0 \subsetneq F_1 \subsetneq F_2 \subsetneq \cdots \subsetneq F_k = E \, .
\]
A basis is now an nbc-basis if and only if its associated maximal chain of flats has a very specific structure.

\begin{lemma}[{\cite[(7.30) (7.31)]{Bjorner:1992}}] \label{lem:nbcBasisToFlat}
Let $B$ be a basis of an ordered matroid $M$ and $\cF(B)$ its associated maximal chain of flats.
Then $B$ is an nbc-basis if and only if $b_i = \min(F_i)$ for each $i\in \{0,\ldots, r(M)\}$.
\end{lemma}

Our main theorem for this subsection will be the following upper bound on the realisation number in terms of nbc-bases.

\begin{theorem} \label{thm:nbc+bases+upper+bound}
Let $G$ be a minimally 2-rigid graph and let $M_G$ be its graphic matroid.
Then
\[
2c_2(G) \leq \nbc{M_G}.
\]
\end{theorem}

A key lemma for deriving bounds will be the following:

\begin{lemma} \label{lem:nbc+basis+intersection}
    Let $M, N$ be two loopless matroids on the totally ordered ground set $E$ such that $r(M) + r(N) = |E| + 1$.
    For each basis $B$ of $M$, write $B':= E \setminus B \cup \min(E)$.
    If $B$ and $B'$ are nbc-bases of $M$ and $N$ respectively, then $(\cF(B), \cF(B')) \in \Delta(M) \times \Delta(N)$ is an intersecting arboreal pair.
\end{lemma}

\begin{proof}
    Without loss of generality, let $1 = \min(E)$, i.e. $B' = E \setminus B \cup \{1\}$.
    Write $\cF := \cF(B)$ and $\cH := \cF(B')$.
    To be an intersecting arboreal pair, we must show that $\Gamma_{\cF,\cH}$ satisfies the conditions of \Cref{prop:flat+cone+intersect}.

    As $\cF$ and $\cH$ are maximal, $\Gamma_{\cF,\cH}$ has $r(M) + r(N) = |E| + 1$ vertices and $|E|$ edges; thus, if we can show $\Gamma_{\cF,\cH}$ is connected then it is a tree.
    Suppose for the sake of a contradiction that $\Gamma_{\cF,\cH}$ is not connected, and let $A$ be a connected component not containing the edge $e(1)$. Note that $\Gamma_{\cF, \cH}$ has no isolated vertices, hence $A$ contains an edge.
    Recall that there is a natural bijection between the edges $E$ and the edges of the intersection graph $\Gamma_{\cF, \cH}$ given by $i \mapsto e(i)$. This bijection naturally induces an ordering on the edges of $\Gamma_{\cF, \cH}$ by $e(i) < e(j)$ if $i \prec j$.

    Let $e(a) = (\widetilde{F}_i, \widetilde{H}_j)$ be the smallest edge in $A$.
    Then, by definition, we have that $a \in \widetilde{F}_{i} \cap \widetilde{H}_{j}$. Observe that any other element $a' \in \widetilde{F}_{i} \cup \widetilde{H}_{j}$ corresponds to an edge $e(a')$ that is incident to either $\widetilde{F}_i$ or $\widetilde{H}_j$, hence it belongs to $A$. So, by the assumption of minimality, we have that $e(a) \le e(a')$ and so $e(a)$ is the smallest element of both $\widetilde{F}_i$ and $\widetilde{H}_j$.
    As $B$ and $B'$ are both nbc-bases, by \Cref{lem:nbcBasisToFlat}, we have that $a \in B \cap B'$.
    But $a\succ 1$, in particular $a \neq 1$,  giving a contradiction. So we have shown that $\Gamma_{\cF, \cH}$ is a tree.

    We next show that every path $\widetilde{F}_i \rightarrow \widetilde{F}_j$ is $\cF$-maximal for all $1 \leq i < j \leq r(M)$.
    We prove a stronger property: the first edge in the path is always the largest.

    First consider the case where $j = r(M)$.
    For each $a \in E \setminus \{1\}$, we orient each edge $e(a) = (\widetilde{F}_k, \widetilde{H}_\ell)$ as follows:
    \[
    \begin{cases}
    \widetilde{F}_k \rightarrow \widetilde{H}_\ell &\text{if } \min(\widetilde{F}_k) \succ \min(\widetilde{H}_\ell), \\
    \widetilde{H}_\ell \rightarrow \widetilde{F}_k &\text{if } \min(\widetilde{H}_\ell) \succ \min(\widetilde{F}_k). \\
    \end{cases}
    \]
    We orient the edge $e(1) = (\widetilde{F}_{r(M)}, \widetilde{H}_{r(M)})$ in the direction $\widetilde{H}_{r(M)} \rightarrow \widetilde{F}_{r(M)}$.

    We now show that every vertex has an outgoing edge except $\widetilde{F}_{r(M)}$.
    For all $\widetilde{F}_k$ with $k \neq r(M)$, the edge $e({a})$ where $a = \min(\widetilde{F}_k)$ must be outgoing.
    If it is incoming, then we have $a = \min(\widetilde{F}_k) \preceq \min(\widetilde{H}_\ell) \preceq a$, hence we have $a = \min(\widetilde{H}_\ell)$. So by \Cref{lem:nbcBasisToFlat}, we have $a \in B \cap B'$.
    By the construction of $B'$, it follows that $a=1$. By the construction of the flag $\cF$, we have $1 = \min(B) \in \widetilde{F}_{r(M)}$. So $k = r(M)$, which is a contradiction. The same argument holds for the vertices $\widetilde{H}_k$, except in the case $k = r(M)$ where the edge $e(1)$ is outgoing. So we have shown that every vertex of $\Gamma_{\cF, \cH}$, except $\widetilde{F}_{r(M)}$, has an outgoing edge.

    As $\Gamma_{\cF,\cH}$ is an oriented tree with a unique sink $\widetilde{F}_{r(M)}$, the unique path $\widetilde{F}_i \rightarrow \widetilde{F}_{r(M)}$ must respect this orientation.
    Moreover, the edges in a path are strictly decreasing by the orientation definition, implying the first edge is the largest.

    For the case where $j \neq r(M)$, the path $\widetilde{F}_i \rightarrow \widetilde{F}_j$ is the symmetric difference of the paths $\widetilde{F}_i \rightarrow \widetilde{F}_{r(M)}$ and $\widetilde{F}_j \rightarrow \widetilde{F}_{r(M)}$.
    By the above, the largest edge of each of these paths is their first edge and they are outgoing. So their largest edges are $e(a)$ and $e(b)$ respectively, where
    \[
    a = \min(\widetilde{F}_{i}) \succ \min(\widetilde{F}_{j}) = b.
    \]
    So we have shown that the first edge of the path is the largest. This concludes the proof that $\Gamma_{\cF, \cH}$ is $\cF$-maximal. The proof that $\Gamma_{\cF, \cH}$ is $\cH$-maximal is identical.
\end{proof}

The key step in this upper bound is the following result on the degree of $-\Trop(M_G)$.
This can be deduced via the machinery in~\cite{AHK:18}, but we give a self-contained proof that avoids technicalities.

\begin{proposition}\label{prop:nbc+degree}
    Let $M$ be a loopless matroid of rank $k$ on $m$ elements.
    Then
    \begin{equation}\label{eq:nbc+degree}
        (-\Trop(M)) \cdot \Trop(U_{m,m-k+1}) \cdot \Trop(y_\epsilon - 1) = \nbc M
    \end{equation}
    where $U_{m,m-k+1}$ is the uniform matroid on $m$ elements of rank $m-k+1$, and $\epsilon$ is an arbitrary element of the ground set of $M$.
\end{proposition}

\begin{proof}
    Let $[m]$ be the ground set of $M$. Choose any generic and rapidly increasing vector $\alpha \in \mathbb{R}^m$ such that  $\alpha_1 < \cdots < \alpha_m$.
    By \Cref{prop:intersecting+arboreal+pair+invariant},
    the tropical intersection product in \eqref{eq:nbc+degree} is equal to the number of intersecting arboreal pairs of $U_{m,m-k+1},M$.
    It suffices to show that there exists a 1-to-1 map between the nbc-bases of $M$ and the intersecting arboreal pairs of $U_{m,m-k+1},M$.
    We do so by proving the following:
    the pair $(\mathcal{F},\mathcal{H})$ is an intersecting arboreal pair of $U_{m,m-k+1},M$ if and only if $\cH = \cF(B)$ for some nbc-basis $B$ of $M$. Given a basis $B$ of $M$, we write $\cH = \cF(B)$ and $\cF = \cF(B')$ for the chains of flats defined by $B$ and $B' := ([m] \setminus B) \cup \{1\}$.

    Assume that $B$ is an nbc-basis of $M$ and let $\cH = \cF(B)$ and $\cF = \cF(B')$ be the chains constructed above.
    Since $|B'| = m - k +1$,
    it is a basis of $U_{m,m-k+1}$.
    Suppose that $B'$ contains a broken circuit $C \setminus \{\min(C)\}$ for some circuit $C$ of $U_{m, m-k+1}$. Since the circuits of $U_{m, m-k+1}$ have size $m-k+2$, it follows that $B' = C \setminus \min(C)$. Since $1 \in B'$, we must have $1 \succ \min(C)$, which is a contradiction. So $B'$ is an nbc-basis, and by \Cref{lem:nbc+basis+intersection}, we have that $(\mathcal{F},\mathcal{H})$ is an intersecting arboreal pair.

    Now suppose that $(\mathcal{F},\mathcal{H})$ is an intersecting arboreal pair.
    As the flats of $U_{m,m-k+1}$ are all subsets of $[m]$ of size at most $(m-k)$ of $[m]$, as well as $[m]$, we have
    \[
        \widetilde{F}_i = F_i \setminus F_{i-1} = \{a_i\} \quad \text{for all } 1 \leq i \leq m-k \, , \quad \widetilde{F}_{m-k+1} = [m] \setminus \{a_1, \dots, a_{m-k}\} \,
    \]
    for some $a_1,\ldots,a_{m-k} \in [m]$.
    No two vertices of $\Gamma_{\cF, \cH}$ can be connected by two or more edges since $\Gamma_{\cF, \cH}$ is a tree.
    As $|\widetilde{F}_{m-k+1}| = k$ and there are exactly $k$ vertices on the $\cH$-part of the graph $\Gamma_{\cF, \cH}$, there is an edge between $\widetilde{F}_{m-k+1}$ and $\widetilde{H}_{j}$ for each $j \in [k]$.
    This implies $|\widetilde{F}_{m-k+1} \cap \widetilde{H}_{j}| = 1$ for each $j \in [k]$, and we write $b_j \in \widetilde{H}_{j}$ for the unique element of $\widetilde{F}_{m-k+1} \cap \widetilde{H}_{j}$.
    We define $B = \{b_1, \dots, b_k\}$ and $B' = ([m]\setminus B) \cup \{1\}=\{a_1,\ldots,a_{m-k},1\}$.

    Next we show that $\cH = \cF(B)$.
    Since $r(H_{i-1} \cup b_i) = r(H_i) = r(H_{i-1})+1$ for each $i \in [k]$ and $\cH$ is maximal (and hence contains $k+1$ flats),
    we have that $B$ is a basis.
    For each $j < \ell$, the vertices $\widetilde{H}_j, \widetilde{H}_\ell$ are both adjacent to $\widetilde{F}_{m-k+1}$. Hence, the unique path $\widetilde{H}_j \rightarrow \widetilde{H}_\ell$ consists of the two edges $\{e(b_j), e(b_\ell)\}$. By $\cH$-maximality, we deduce that $b_j \succ b_\ell$ and so $b_1 \succ \cdots \succ b_k$. Therefore $\cH = \cF(B)$.

    We now show that $B$ is an nbc-basis.
    Choose any $j \in [k]$.
    If $a_i \in \widetilde{H}_j$, then, by the argument in the previous paragraph, the unique path $\widetilde{F}_i \rightarrow \widetilde{F}_{m-k+1}$ is $e(a_i), e(b_j)$.
    So, by $\cF$-maximality, we have $a_i \succ b_j$.
    Since each element of $\widetilde H_j$, except $b_j$, is given by $a_i$ for some $i$, we deduce that $b_j = \min(H_j)$ for each $j \in [k]$.
    Hence $B$ is an nbc-basis by \Cref{lem:nbcBasisToFlat}.

    Finally, we show that $\cF = \cF(B')$.
    Since $B$ contains 1 (as all nbc-bases must),
    the set $B'$ has exactly $m-k+1$ elements and $1 \in \widetilde{F}_{m-k+1}$.
    By the above, the unique path $\widetilde{F}_i \rightarrow \widetilde{F}_{m-k+1}$ begins with the largest edge, namely $e(a_i)$.
    The unique path from $\widetilde{F}_i \rightarrow \widetilde{F}_{j}$ is the symmetric difference of the paths from $\widetilde{F}_i \rightarrow \widetilde{F}_{m-k+1}$ and $\widetilde{F}_{j} \rightarrow \widetilde{F}_{m-k+1}$.
    If $1 \leq i < j \leq m-k$, then, by $\cF$-maximality, we have $a_i \succ a_j$, hence $a_1 \succ \cdots \succ a_{m-k} \succ 1$.
    So $\cF = \cF(B')$.

    We have shown that if $(\cF, \cH)$ is an intersecting arboreal pair then there exists an nbc-basis $B$ such that $\cH = \cF(B)$ and $\cF = \cF(B')$ where $B' = ([m] \setminus B)\cup\{1\}$. This finishes the proof.
\end{proof}

We prove some results that bound the tropical intersection product.
Given a weighted polyhedral complex $\Sigma$ and a point $w$ in its support, the \emph{star} of $\Sigma$ at $w$ is the weighted polyhedral complex
\[
\mathrm{star}_w(\Sigma) = \bigcup_{w \in \sigma \in \Sigma} \sigma_w \, , \quad \sigma_w := \left\{\lambda(v - w) \mid v \in \sigma \, , \, \lambda \geq 0\right\} \, ,
\]
where the cone $\sigma_w$ has multiplicity $\mult_\Sigma(\sigma)$.
Intuitively, $\mathrm{star}_w(\Sigma)$ is $\Sigma$ viewed locally around $w$.

\begin{lemma}
  \label{lem:IntersectionNumberStar}
  Let $\Sigma_1,\Sigma_2$ be two balanced polyhedral complexes of complementary dimension in $\RR^n$ and $w\in\RR^n$ in the support of $\Sigma_1$.  Then
  \[\Sigma_1\cdot\Sigma_2\geq\mathrm{star}_w(\Sigma_1)\cdot\Sigma_2.\]
\end{lemma}

\begin{proof}
  Without loss of generality, we may assume that $w=\mathbf{0} = (0,\ldots,0)$.
  For $t>0$, let $t\cdot\Sigma_1$ be the balanced polyhedral complex with polyhedra $t\cdot\sigma$, where $\sigma\in\Sigma_1$ and $t\cdot (\ldots)$ denotes linear scaling by $t$, and multiplicities $\mult_{t\cdot\Sigma}(t\cdot\sigma)=\mult_{\Sigma_1}(\sigma)$.  Observe that locally the support of the scaled complex $t \cdot \Sigma_1$ is given by a translation of the support of $\Sigma_1$. More precisely,
  for each point $x$ contained in the support of $t \cdot \Sigma_1$ there exists an open set $U \subseteq \RR^n$ and a vector $v \in \RR^n$ such that $x \in U$ and $U \cap (t \cdot \Sigma_1) = v + U \cap \Sigma_1$ (as sets).
  So, by \cref{lem:tropicalIntersectionProductTranslationInvariant}, we have $(t\cdot \Sigma_1)\cdot\Sigma_2=\Sigma_1\cdot\Sigma_2$ for all $t>0$. Note that $t\cdot \Sigma_1$ converges point-wise to $\mathrm{star}_\mathbf{0}(\Sigma_1)$ as $t$ goes to $\infty$ and the stable intersection points of $(t\cdot\Sigma_1)\wedge\Sigma_2$ vary continuously in $t$. We now consider what happens to these intersection points as $t$ increases.

  Let 
  \begin{equation*}
       s_t\coloneqq\min\{\|u\|\mid u\in \sigma, ~ \sigma\in t\cdot\Sigma_1, ~\mathbf{0}\notin\sigma\}
  \end{equation*}
 denote the minimal distance between the point $\mathbf{0}$ and all polyhedra of $t\cdot \Sigma_1$ not containing $\mathbf{0}$, and let $B_t$ denote the ball around $\mathbf{0}$ of radius $s_t$.  Then $t\cdot \Sigma_1$ and $(t+t')\cdot\Sigma_1$ coincide inside $B_t$ for all $t'>0$, and thus $x\in (t\cdot \Sigma_1\wedge\Sigma_2)\cap B_t$ implies $x\in\mathrm{star}_\mathbf{0}(\Sigma_1)\wedge\Sigma_2$.  
 Moreover, $B_t$ converges to $\RR^n$ as $t$ goes to infinity.
  Then, as $t$ goes to infinity, any intersection point of $(t\cdot\Sigma_1)\wedge\Sigma_2$ either falls into $B_t$, becoming an intersection point of $\mathrm{star}_\mathbf{0}(\Sigma_1)\cdot\Sigma_2$, or diverges to infinity. This shows the claim.
\end{proof}

To prove \Cref{thm:nbc+bases+upper+bound}, we need to make use of an alternative characterisation of tropical linear spaces detailed in \cite{Speyer:2008} and \cite[Section 4.4]{MaclaganSturmfels15}.
Given some $p \in (\RR \cup \{\infty\})^{\binom{[n]}{k}}$ and $w \in \RR^n$, define
\begin{align} \label{eq:initial+matroid}
\cB(M_{p,w}) := \left\{ B' \in \binom{[n]}{k} \, \Big| \, \min_{B \in \binom{[n]}{k}} \bigg(p_B - \sum_{i \in B} w_i  \bigg) \text{ attained at } B' \right\} .    
\end{align}
We say $p \in (\RR \cup \{\infty\})^{\binom{[n]}{k}}$ is a \emph{tropical Pl\"ucker vector} if the set system $\cB(M_{p,w})$ forms the bases of a matroid $M_{p,w}$ for all $w \in \RR^n$.
We define its associated \emph{tropical linear space} to be
\[
L_p = \left\{w \in \RR^n \mid M_{p,w} \text{ loopless matroid}\right\} .
\]
$L_p$ is a balanced polyhedral complex where each maximal cell has multiplicity one.

This aligns with our existing notion of a (realisable) tropical linear space: if $I \subseteq K[x_1^\pm, \dots, x_n^\pm]$ is a $k$-dimensional linear ideal, there exists a unique tropical Pl\"ucker vector $p \in (\RR \cup \{\infty\})^{\binom{[n]}{k}}$ up to addition of scalars such that $\Trop(I) = L_p$.
The converse is not true: there exist tropical linear spaces $L_p$ that are not realisable as the tropicalisation of a linear ideal.
For example, if $M$ is a non-realisable matroid then \cref{lem:bergmanFan} implies there is no linear ideal $I$ such that $\Trop(I)$ is the Bergman fan of $M$.

In addition, we require the notion of a \textit{recession fan}. Suppose that $\sigma = \{ x \in \RR^n \mid  Ax \le b\}$ is a polyhedron given by its half-space description for some matrix $A \in \RR^{d \times n}$ and $b \in \RR^d$. Then the recession cone of $\sigma$ is defined as $\rec(\sigma) = \{x \in \RR^n \mid Ax \le \mathbf{0} \}$. More generally, if $\Sigma$ is a polyhedral complex, then $\rec(\Sigma)$ is the union of all cones $\rec(\sigma)$ taken over polyhedra $\sigma \in \Sigma$. 
The set $\rec(\Sigma)$ 
can always be given the structure of a polyhedral fan. 
Moreover, if $\Sigma'$ is another polyhedral complex with the same support as $\Sigma$, then the recession fans $\rec(\Sigma)$ and $\rec(\Sigma')$ have the same support, hence the recession fan is well-defined on supports. This allows us to talk about \textit{the} recession fan of a tropical variety $X$, by which we mean the well-defined unique support of the recession fan given by some choice a polyhedral complex structure on $X$. The multiplicities given by $\mult_{\rec(X)}(\sigma)=\sum_{\rec(\sigma')=\sigma}\mult_X(\sigma')$ makes it a balanced fan.

\begin{lemma}\label{lem:degrees}
    Let $\Sigma$ be a pure balanced polyhedral complex in $\RR^n$ of dimension $1 \leq s \leq n-1$ and let $L_q$ a tropical linear space with tropical Pl\"{u}cker vector $q\in \RR^{\binom{[n]}{n-s}}$. Then $\Sigma \cdot L_q = \Sigma \cdot \Trop(U_{n,n-s})$.
\end{lemma}

\begin{proof}
    As $q \in \RR^{\binom{[n]}{n-s}}$ has only finite values, its recession fan is $\rec(L_q)=\Trop(U_{n,n-s})$ by \cite[Theorem 4.4.5]{MaclaganSturmfels15}, and that $\Trop(U_{n,n-s})=\rec(\Trop(U_{n,n-s}))$ as $\Trop(U_{n,n-s})$ is already a fan. Thus $\rec(L_q)=\rec(\Trop(U_{n,n-s}))$.
    By \cite[Theorem 5.7]{AHR:16}, if $X$ and $Y$ are polyhedral complexes with codimension complementary to that of $\Sigma$, then $\rec(X) = \rec(Y)$ implies that $\Sigma \cdot X = \Sigma \cdot Y$. Thus $\Sigma \cdot L_q = \Sigma \cdot \Trop(U_{n,n-s})$.
\end{proof}

\begin{lemma}\label{lem:Bergman+degen}
    Let $\Sigma$ be a pure balanced polyhedral complex in $\RR^n$ of dimension $1 \leq s \leq n-1$, and let $M$ be a loopless matroid on ground set $[n]$ with $r(M) = n-s$.
    Then
    \[
    \Sigma \cdot \Trop(M) \leq \Sigma \cdot \Trop(U_{n,n-s}) \, .
    \]
\end{lemma}

\begin{proof}
Throughout the proof, we denote by $\mathbf 0 = (0,0, \dots, 0)$ the all-zeros vector of appropriate length. We deform $\Trop(U_{n,n-s})$ to the tropical linear space $L_q$ described by the vector
    \[
    q \in \RR^{\binom{[n]}{n-s}} 
    \quad \text{with} \quad
    q_B = r_M([n]) - r_M(B) \quad \text{for each } B \in \binom{[n]}{n-s},
    \]
    where $r_M$ denotes the rank function of the matroid $M$.
    It is a tropical Pl\"ucker vector by \cite[Lemma 27]{js17}, hence $L_q$ is a well-defined tropical linear space.
    By \Cref{lem:degrees}, we see that $\Sigma \cdot L_q = \Sigma \cdot \Trop(U_{n,n-s})$. 
    As $q_B \geq 0$ with equality if and only if $B$ a basis of $M$, 
    we have $M_{q,\mathbf 0} = M$, which is a loopless matroid, hence $\mathbf{0} \in L_q$. By \cite[Corollary 4.4.8]{MaclaganSturmfels15}, the star $\mathrm{star}_{\mathbf{0}}(L_q)$ is equal to $\Trop(M)$.
    By \Cref{lem:IntersectionNumberStar}, we have $\Sigma \cdot \mathrm{star}_{\mathbf{0}}(L_q) \leq \Sigma \cdot L_q$. Putting this together gives 
    \[
    \Sigma \cdot \Trop(M) =\Sigma \cdot \mathrm{star}_{\mathbf{0}}(L_q) \leq \Sigma \cdot L_q = \Sigma \cdot \Trop(U_{n,n-s}) \, .\qedhere
    \]
\end{proof}

\begin{proof}[Proof of \Cref{thm:nbc+bases+upper+bound}]
    Fix an edge $\epsilon \in E$.
    We first note that
    \begin{equation*}
        (-\Trop(M_G))\wedge\Trop(y_{\epsilon} - 1) = -\Trop(M_G)|_{y_{\epsilon} = 0}:
    \end{equation*}
    the proof of this is identical to that given in the proof of \Cref{prop:intersecting+arboreal+pair+invariant}.
    We now have that
    \begin{equation*}
        \begin{array}[b]{rcl}
            2c_2(G) &\overset{\text{Thm. }\ref{thm:main}}=& (-\Trop(M_G)|_{y_{\epsilon} = 0}) \cdot (\Trop(M_G)) \\
            &\overset{\text{Lem. \ref{lem:Bergman+degen}}}\leq& (-\Trop(M_G)|_{y_{\epsilon} = 0}) \cdot (\Trop(U_{2n-3,n-1}) \\[2mm]
            &=& (-\Trop(M_G)) \cdot \Trop(U_{2n-3,n-1}) \cdot \Trop(y_{\epsilon} - 1) \\
            &\overset{\text{Prop. \ref{prop:nbc+degree}}}=& \nbc{M_G}.
        \end{array}\qedhere
    \end{equation*}
\end{proof}

\begin{example}\label{example: nbc bases K4-e}
Consider the graph $G$ in \Cref{example: intersecting arboreal pairs} and \Cref{fig:intersecting_arboreal_pairs} and order the edges $1 > 2 > 3 > 4 > 5$.
The broken circuits of $G$ under this ordering are $\{12,34,123\}$.
Enumerating through all spanning trees yields that there are four nbc-bases
$\{135, 145, 235, 245\}$.
In this case, \Cref{thm:nbc+bases+upper+bound} states that $c_2(G) \leq 2$.
In fact, the realisation number of $G$ is $c_2(G) = 2$ and so the nbc-bound is equal to the 2-realisation number.
\end{example}

\begin{figure}[t]
\centering
\includegraphics[width=0.3\textwidth]{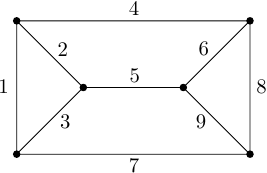}
\caption{The 3-prism, with edges labelled for \Cref{example: nbc bases prism}.}
\label{fig:prism}
\end{figure}

\begin{example}\label{example: nbc bases prism}
The $3$-prism is the graph $G = (V, E)$ displayed in \Cref{fig:prism}. 
We order the edges $1 > 2 > \cdots > 9$. The broken circuits of $G$ are given by
\[
    12,\, 68, \, 147, \, 245, \, 357, \,
    1257,\, 1345,\, 1467,\, 2347,\, 2458,\, 3567,\, 12567,\, 13458,\, 23467.
\]
For instance, the edges 
$6$, $8$, $9$ form a cycle. Since $9$ is the minimum edge label, we have that $68$ is a broken circuit. The graphic matroid of $G$ has $75$ bases, of which $26$ are nbc-bases:
\[
\begin{gathered}
    13469, 13489, 13569, 13589,
    13679, 13789, 14569, 14589,
    15679, 15789, 23469, 23489, 23569, \\
    23589, 23679, 23789, 24679,
    24789, 25679, 25789, 34569,
    34589, 34679, 34789, 45679, 45789.
\end{gathered}
\]
From \Cref{thm:nbc+bases+upper+bound}, we can deduce that $c_2(G) \leq 13$.
The actual realisation number of $G$ is $c_2(G) =12$, and so the number of nbc-bases gives a strict upper bound on the realisation number.
\end{example}

As a corollary of \Cref{thm:nbc+bases+upper+bound}, we note that the number of nbc-bases of $M_G$ is closely connected to a number of other graph and matroid invariants.
We close this subsection by recalling them and stating the bound on the realisation number in terms of them.
For further details and proofs of the following claims, we refer the reader to \cite{monaghan2022handbook}.

Let $M$ be a matroid on ground set $E$ and rank function $r$.
The \emph{Tutte polynomial} of $M$ is defined as
\[
T(M; x, y) := \sum_{A \subseteq E} (x-1)^{r(E) - r(A)}(y-1)^{|A| - r(A)} \, .
\]
The Tutte polynomial is the universal matroid invariant for matroids under deletion and contraction.
Moreover, we can obtain the number of nbc-bases as the evaluation of the Tutte polynomial $\nbc M = T(M; 1, 0)$.

The \emph{characteristic polynomial} of $M$ is defined as
\[
\Chi_M(\lambda) := \sum_{A \subseteq E} (-1)^{|A|} \lambda^{r(E) - r(A)} = (-1)^{r(E)} T(M; 1-\lambda, 0) \, .
\]
In particular, by evaluating the characteristic polynomial at zero we deduce that $\nbc M = (-1)^{r(E)}\Chi_M(0) = |\Chi_M(0)|$ is the absolute value of the constant term of $\Chi_M(0)$.

The \emph{chromatic polynomial} $P(G, \lambda)$ of a graph $G$ is the unique polynomial whose evaluation at any integer $k \geq 0$ is the number of $k$-colourings of $G$.
When $G$ is connected, the chromatic polynomial of $G$ is related to the characteristic polynomial of $M_G$ by $P(G; \lambda) = \lambda \cdot \Chi_M(\lambda)$.
As such, $\nbc M$ is the absolute value of the coefficient of $\lambda$ in $P(G; \lambda)$.

\begin{corollary} \label{cor:alt+bound}
Let $G$ be a minimally 2-rigid graph with graphic matroid $M_G$.
Then $2c_2(G)$ is bounded above by the following equivalent values:
\begin{enumerate}
\item the value of the Tutte polynomial evaluation $T(M_G; 1, 0)$,
\item the constant term (up to sign) of the characteristic polynomial $\Chi_{M_G}(\lambda)$,
\item the coefficient of the linear term (up to sign) of the chromatic polynomial $P(G; \lambda)$ of $G$.
\end{enumerate}
\end{corollary}

There are additional characterisations of nbc-bases that one can derive in terms of acyclic orientations of $G$ with a unique source, or in terms of maximal $G$-parking functions of the graph.
We refer the interested reader to~\cite[Theorem 4.1]{BCT:2010} for these.

\subsection{Realisation bases}
Jackson and Owen conjectured the following lower bound for the realisation number.
\begin{conjecture}[\cite{JO19}] \label{conj:lowerbound}
Every minimally 2-rigid graph $G$ with $n$ vertices satisfies $c_2(G) \geq 2^{n-3}$.
\end{conjecture}

As evidence towards this, they proved the conjecture for planar minimally 2-rigid graphs.
Moreover, it has been verified for $n \leq 12$ in~\cite{cggkls}.
Despite this, in general we have no better lower bound than $c_2(G) \geq 2$ for an arbitrary minimally 2-rigid graph.

In this section, we utilize some of the theory developed in the previous section to give a lower bound on the realisation number in terms of nbc-bases.

\begin{definition}
Let $G = ([n], E)$ be a minimally 2-rigid graph and $(M_G, \prec)$ its graphic matroid with some fixed total order $\prec$ on $E$.
A \emph{realisation basis} is an nbc-basis $B$ of $(M_G, \prec)$ such that $E \setminus B \cup \min(E)$ is also an nbc-basis.
\end{definition}

Realisation bases naturally come in pairs, as $B$ is a realisation basis if and only if $E \setminus B \cup \min(E)$ also is.
Unlike with nbc-bases, the number of realisation bases of the matroid $(M_G, \prec)$ is highly dependent on the total order $\prec$.
As such, we shall always discuss the number of realisation bases with respect to a specific order.

\begin{corollary} \label{cor:lower+bound}
    Let $G = ([n], E)$ be a minimally 2-rigid graph and $(M_G, \prec)$ its graphic matroid with some fixed total order $\prec$ on $E$.
    Then
    \[
    2c_2(G) \geq \# \text{ of realisation bases of } (M_G, \prec) \, .
    \]
\end{corollary}

\begin{proof}
    Applying \Cref{lem:nbc+basis+intersection} in the case where $M = N = M_G$ tells us that the pair
    \begin{equation*}
        \cF(B), \quad \cF\big((E \setminus B) \cup \min(E) \big)
    \end{equation*}
    form an intersecting arboreal pair for each realisation basis $B$ of $(M_G, \prec)$.
    It then follows from \Cref{thm:arboreal} that $c_2(G)$ is lower-bounded by this count.
\end{proof}

\begin{example}
Consider the graph $G = K_4^-$ from \Cref{example: nbc bases K4-e}. 
Keeping the same ordering on the edges, we see that all of the nbc-bases are also realisation bases.
For example, the nbc-basis $B = 135$ is paired with the nbc-basis $B' = [5] \setminus E \cup\{5\} = 245$, implying they are both realisation bases.
As such, the lower bound from \Cref{cor:lower+bound} is an equality for this choice of ordering.
We emphasise that other choices of ordering may not achieve this lower bound.

\end{example}

\begin{example}
Now let $G$ be the 3-prism from \Cref{example: nbc bases prism}.
Keeping the same ordering on the edges, we see that there are only 16 realisation bases:
\[
\begin{gathered}
    13469, 13489, 13569, 13589,
    14569, 14589, 15679, 15789, \\
    23469, 23489, 23679, 23789,
    24679, 24789, 25679, 25789.
\end{gathered}
\]
For example, the nbc-basis $B = 13679$ is not a realisation basis as $B' = E \setminus B \cup \{9\} = 24589$ is not an nbc-basis: in this case, it is not even a basis.
As such, the lower bound from \Cref{cor:lower+bound} gives $c_2(G) \geq 8$ for this choice of ordering.
A computer check shows that among all total orderings of the edges of the $3$-prism, this ordering gives the largest possible number of realisation bases, and so the number of realisation bases gives a strict lower bound on the realisation number for every ordering of the edges.
\end{example}

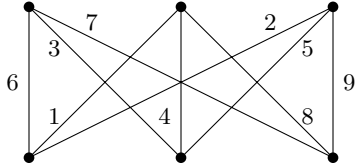
\begin{figure}[t]
\centering
\begin{tikzpicture}[scale=2,every node/.style={font=\footnotesize}]
  \coordinate (v11) at (1,1);
  \coordinate (v21) at (2,1);
  \coordinate (v31) at (3,1);
  \coordinate (v12) at (1,2);
  \coordinate (v22) at (2,2);
  \coordinate (v32) at (3,2);

  \draw
  (v11) -- node[anchor=east] {$6$} (v12)
  (v11) -- node[anchor=east,pos=0.275] {$1$} (v22)
  (v11) -- node[anchor=south east,pos=0.825,inner sep=1pt] {$2$} (v32)
  (v21) -- node[anchor=east,pos=0.725] {$3$} (v12)
  (v21) -- node[anchor=east,pos=0.275] {$4$} (v22)
  (v21) -- node[anchor=west,pos=0.725] {$5$} (v32)
  (v31) -- node[anchor=south west,pos=0.825,inner sep=1pt] {$7$} (v12)
  (v31) -- node[anchor=west,pos=0.275] {$8$} (v22)
  (v31) -- node[anchor=west] {$9$} (v32);

  \fill
  (v11) circle (1pt)
  (v21) circle (1pt)
  (v31) circle (1pt)
  (v12) circle (1pt)
  (v22) circle (1pt)
  (v32) circle (1pt);
\end{tikzpicture}
\caption{The complete bipartite graph $K_{3,3}$ with edges labelled for \Cref{ex:realisation+basis+below+conj}.}
\label{fig:K33}
\end{figure}

\begin{example}\label{ex:realisation+basis+below+conj}
    We note that even for relatively small graphs, the lower bound from realisation bases may not reach the conjectured lower bound of $2^{n-3}$ from \Cref{conj:lowerbound}.
    Consider the complete bipartite graph $K_{3,3}$ with edge labels given in \Cref{fig:K33}.
    Its realisation number is equal to the conjectured lower bound of $8$.
    Under the usual ordering $1 > \cdots > 9$, there are 14 realisation bases of $K_{3,3}$:
    \[
    \begin{gathered}
    12579, 13569, 13689, 13789, 14579, 14569, 15789, \\
    34689, 24789, 24579, 24569, 23689, 23789, 23469.
    \end{gathered}
    \]
    \Cref{cor:lower+bound} gives a lower bound of $c_2(K_{3,3}) \geq 7$.
    Moreover, a computer search shows us that no other ordering of the edges gives rise to more than 14 realisation bases, hence the lower bound offered by realisation bases does not attain the conjectured lower bound.
\end{example}

\section{Concluding remarks}

We end the paper with some comments and avenues for future development.

\subsection{Computations}\label{subsec: computations}

We computed upper bounds for the realisation number given by the mixed volume and by nbc-bases for all minimally 2-rigid graphs on at most $10$ vertices \cite{RigidityTestSuiteM2}. See \Cref{tab: upper bound computations}. In total we have tested $117273$ graphs. There are $100958$ graphs ($\approx 86.1\%$) for which nbc-bases give a strictly better upper bound than the mixed volume. There are $14876$ graphs ($\approx
12.7\%$) for which mixed volumes give a better upper bound than nbc-bases. For the remaining $1439$ graphs ($\approx 1.2\%$), the number of nbc-bases matched the mixed volume bound.

The table shows how far the upper bounds are from the realisation number. The scaled results show how far the bounds are as a multiple of the realisation number. For instance, suppose $G$ is a graph with $100$ complex realisations and whose matroid has $250$ nbc-bases. Then the difference between the realisation number and the number of nbc bases is $150$ and the scaled difference is $(250 - 100)/100 = 1.5$. \Cref{fig: histogram} shows a histogram of the scaled differences for nbc bases and mixed volume.

\begin{table}[]
    \centering
    \begin{tabular}{ccccc}
    \toprule
        &
        mean &  s.d. & mean (scaled) & s.d. (scaled)  \\
    \midrule
       Mixed volume & 1147.63 & 1044.94 & 4.07 & 3.85 \\
       nbc-bases & 337.91 & 159.75 & 1.21 & 0.59 \\
    \bottomrule
    \end{tabular}
    \caption{Difference between upper bounds and realisation numbers in \Cref{subsec: computations}.}
    \label{tab: upper bound computations}
\end{table}

\begin{figure}[t]
    \centering
    \includegraphics[width=0.9\linewidth]{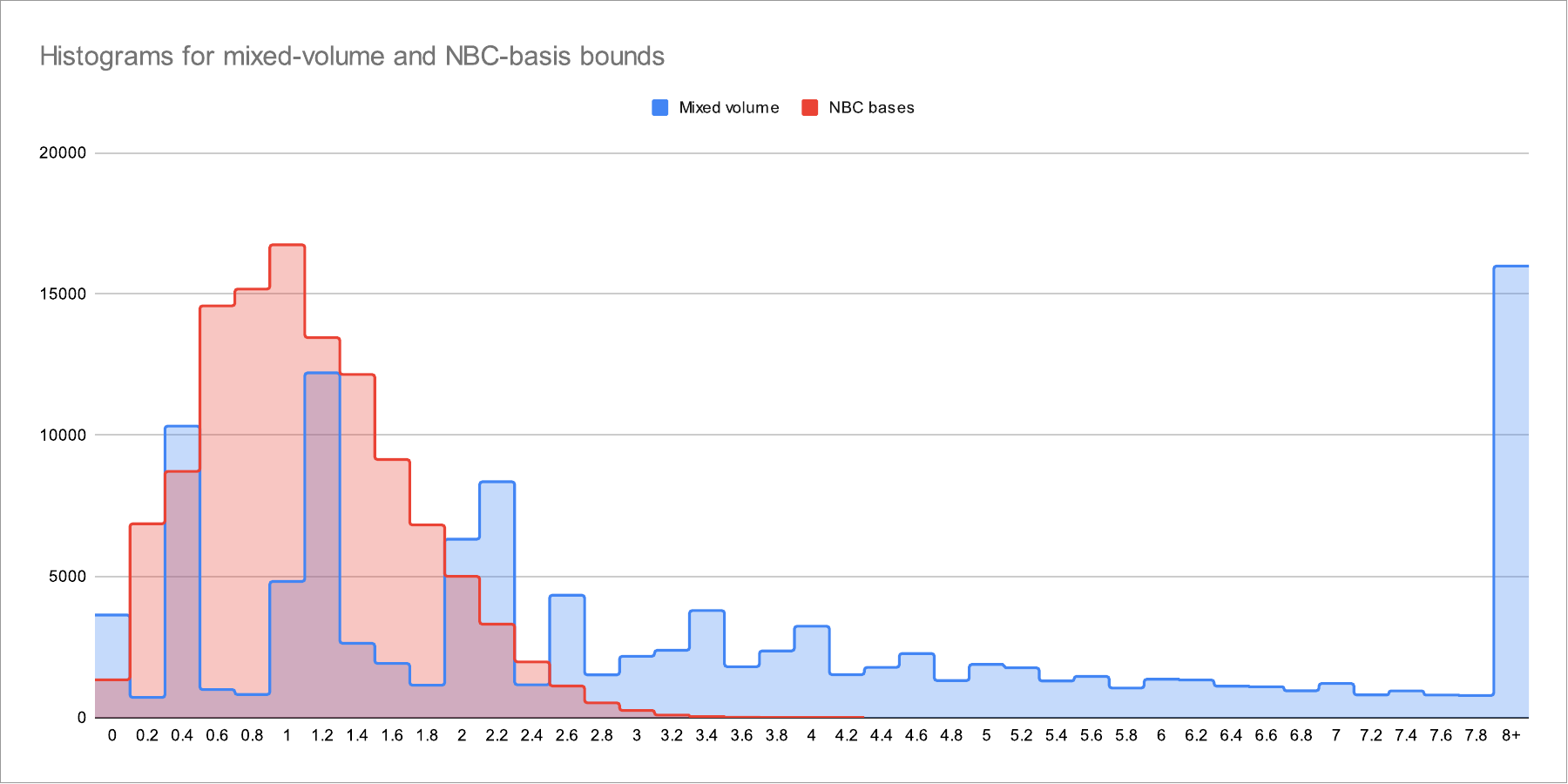}
    \caption{Histogram for the mixed-volume-bound and nbc-basis-bound. The $x$-axis is the scaled difference between the upper bound and realisation number}
    \label{fig: histogram}
\end{figure}

\subsection{Bigraphs}

A \emph{bigraph} is a pair of multigraphs $(G,H)$ that share an edge set $\mathcal{E}$, in the sense that there exist bijections $\phi_G:E(G) \rightarrow \mathcal{E}$ and $\phi_H:E(H) \rightarrow \mathcal{E}$.
The concept of a bigraph was first introduced in \cite{cggkls} in the design of their inductive algorithm for realisation numbers.
Each bigraph $(G,H)$ can be assigned a value $\Lam(G,H) \in \mathbb{Z}_{\geq 0} \cup \{\infty\}$ called the \emph{Laman number} of the pair $(G,H)$.
For the bigraph $(G,G)$ with edges paired up with their copies, they showed that $\Lam(G,G) = 2 c_2(G)$ \cite[]{cggkls}.

A small adaptation to \cite[Lemma 2.16]{cggkls} shows that the Laman number of a bigraph $(G,H)$ is equal to the root count of the following ideal for generic $(a_e)_{e \in \mathcal{E}}, c_{1,1},c_{1,2}$ and any fixed $\epsilon \in \mathcal{E}$:
\begin{equation}\label{eq:bigraph}
    \begin{aligned}
        f_e &\coloneqq y_{e,1}\cdot y_{e,2} - a_{e} &&\text{for } e \in \mathcal{E} \\
        g_{\epsilon} &\coloneqq c_{1,1}y_{\epsilon,1} + c_{1,2}y_{\epsilon,2}, \\
        h_{C,G}&\coloneqq \sum_{(s,t)\in C} y_{st,1} &&\text{for each directed cycle $C$ of } G ,\\
        h_{D,H}&\coloneqq \sum_{(s,t)\in D} y_{st,2} &&\text{for each directed cycle $D$ of } H.
    \end{aligned}
\end{equation}
This similarity to the ideal $I'_E$ given in \Cref{lem:realisationNumberTwoDimensionalEdgeVariables} can be exploited to describe a bigraph analogue for the ideal $I''_E$ given in \Cref{lem:SIAGA2}.
The natural bigraph analogue to \Cref{thm:main} then follows the same proof technique with the graph $G$ used in the ideal $I_1$ replaced by $H$:

\begin{theorem}
  \label{thm:bigraph}
  Let $(G,H)$ be a bigraph with shared edge set $\mathcal{E}$.
  If $\Lam (G,H)$ is a positive integer,
  then for any $\epsilon \in \mathcal{E}$ we have
  \[ \Lam (G,H) = (-\Trop(M_G))\cdot \Trop(M_H)\cdot \Trop(y_{\epsilon}-1). \]
\end{theorem}

\Cref{thm:bigraph} also can be combined with the 2-realisation number algorithm described in \cite{cggkls}:
this then gives an inductive contraction-deletion algorithm for determining the tropical intersection product $(-\Trop(M))\cdot \Trop(N)\cdot \Trop(y_{\epsilon}-1)$ when $M$ and $N$ are both graphical matroids on a common ground set.
It is currently open if such an algorithm exists when $M$ and $N$ are not graphical.

\subsection{Higher dimensions}

Up to \Cref{lem:realisationNumberGeneralEdgeVariables} we worked in arbitrary dimension. However after that we had to restrict to 2-dimensions.
In particular, one of the main `tricks' we implement is to switch the distance between pairs being of the form $x^2 + y^2$ to $xy$ using a linear transformation.
This is no longer possible for $d = 3$ (in fact, for $d \geq 3$) since the corresponding distance equation is of the form $x^2 + y^2 + z^2$,
which is irreducible.

We can tropicalise the equations in \Cref{lem:realisationNumberGeneralEdgeVariables}, but to the best of our knowledge the resulting tropical variety has no nice combinatorial characterisation. It would be interesting to find an analogue of \Cref{thm:main} in higher dimensions.

\subsection{Realisation numbers as matroid invariants}

Recent innovations of intersection theory for matroids have lead to a number of success stories quantifying known matroid invariants as tropical intersection products of Bergman fans and their `flips'.
We have already seen and utilised the result of \cite{AHK:18} where $\nbc M$ is equal to the tropical intersection product of the flipped Bergman fan $-\Trop(M)$ with the Bergman fan of a uniform matroid.
Moreover, \cite{AEP:23} deduced that the $\beta$-invariant $\beta(M)$ of a matroid can be computed as a tropical intersection product of Bergman fans derived from $M$.

Our \Cref{thm:main} has a similar flavour to these results, demonstrating that the 2-realisation number of $G$ can be quantified as the tropical intersection product of two fans from $M_G$.
Moreover, these previous developments provide evidence that the realisation number may have a combinatorial formula in terms of known matroid invariants, and that using the intersection theory of matroids is the method to obtain this.
We hope that such theoretical developments may lead to computational and algorithmic improvements for computing realisation numbers.


\subsection*{Acknowledgements}

We thank Matteo Gallet and Josef Schicho for their helpful conversations.
This project arose from a Focused Research Group funded by the Heilbronn Institute for Mathematical Research (HIMR) and the UKRI/EPSRC Additional Funding Programme for Mathematical Sciences.
S.\,D.\ and J.\,M.\ were supported by HIMR.
D.\,G.\,T.\ was supported by EPSRC grant EP/W524414/1.
A.\,N.\ and B.\,S.\ were supported by EPSRC grant EP/X036723/1.
Y.\,R.\ was supported by UKRI grant MR/S034463/2.

\renewcommand*{\bibfont}{\small}

\printbibliography
\end{document}